\documentclass[final,leqno]{siamltex}
\usepackage{threeparttable}% tables with footnotes
\usepackage{dcolumn}%   decimal-aligned tabular math columns
\newcolumntype{d}{D{.}{.}{-1}}
\usepackage{subfigure}% subcaptions for subfigures
\usepackage[german,english]{babel}
\usepackage{amssymb}
\usepackage{algorithm}
\usepackage{algorithmic}
\usepackage{color}
\usepackage{graphicx}
\usepackage{booktabs}
\usepackage{amsfonts}
\usepackage{soul}
\usepackage{mathrsfs}
\usepackage{mathtools}
\usepackage{amsmath, amssymb}
\usepackage{float}
\usepackage{threeparttable}% tables with footnotes
\usepackage{dcolumn}%   decimal-aligned tabular math columns
\newcolumntype{d}{D{.}{.}{-1}}

\usepackage{enumitem}

\title{Nonlinear Model Reduction via an Adaptive Weighting of Snapshots\thanks{
This research was supported by the Office of Naval Research.}}

\author{Liqian Peng\thanks{Department of Mechanical and
Aerospace Engineering, and Institute for Networked Autonomous
Systems, University of Florida, Gainesville, FL 32611-6250
(liqianpeng@ufl.edu).}
        \and Kamran Mohseni\thanks{Department of Mechanical and
Aerospace Engineering, Department of Electrical and Computer
Engineering, and Institute for Networked Autonomous Systems,
University of Florida, Gainesville, FL 32611-6250
(mohseni@ufl.edu).}}

\begin{document}

\maketitle
\begin{abstract}
In this paper, we propose a new approach to  model reduction of parameterized partial differential equations (PDEs) based on the concept of adaptive reduced bases. The presented approach is particularly suited for large-scale nonlinear systems characterized by parameter variations. Instead of using a global basis to construct a global reduced model, the proposed method approximates the original system by multiple lower-dimensional subspaces. Each localized reduced basis is generated by the SVD of a weighted snapshot ensemble; here, each weighting coefficient is a function of the input parameter. Compared with a global model reduction method, such as the classical POD, the adaptive model reduction method could yield a more accurate  solution with a fixed subspace dimension. Moreover, we combine the adaptive reduced model  with the chord iteration to solve elliptic PDEs in a computationally efficient fashion. The potential of the method for achieving large speedups, while maintaining good accuracy, is demonstrated for both elliptic and parabolic PDEs in a few numerical examples.
\end{abstract}

\begin{keywords}
model reduction; adaptive weighting; chord iteration
\end{keywords}

\pagestyle{myheadings} \thispagestyle{plain} \markboth{LIQIAN PENG
AND KAMRAN MOHSENI}{SYMPLECTIC  MODEL REDUCTION OF HAMILTONIAN SYSTEMS}

\section{Introduction}
\vspace{-2pt}

Due to computing speed barriers, in many engineering applications, direct numerical simulations are so computationally intensive that they either cannot be performed as often as needed. For this reason, during the past several decades, many efforts have been put forward to develop reduced models for time-critical operations such as electrical power grids \cite{Marsden:99e, PetzoldLR:02a},
structural dynamics \cite{AmabiliM:03a} , chemical reaction
systems \cite{KevrekidisI:96a, KevrekidisI:98a}, and CFD-based
modelling and control \cite{Holmes:98a, RowleyCW:04a,
AtwellJA:01a, BergmannM:05a}, to list but a few. The main idea for this kind of model reduction is based on the following: although the state of a complex system is in general represented by a large dimensional space, the linear subspace spanned by solution snapshots actually has a much lower dimension.  To this effect, the proper orthogonal decomposition (POD) with Galerkin projection \cite{Holmes:98a, MooreBC:81a} has been developed to generate lower dimensional surrogates for the original large-scale complex problems.  While POD always looks for a linear subspace instead of its curved submanifold, it is computationally tractable and able to capture the dominant patterns in a nonlinear system.  A typical application of the POD-Galerkin approach involves an offline-online splitting methodology. In the offline stage, the original problem, corresponding to some sampled input parameters, is solved to obtain some solution snapshots. For the online computation, a linear subspace for new input-output evaluation is constructed. This methodology  is very suitable for the real-time or many-query applications to achieve minimal marginal cost per input-output evaluation.

To the best of our knowledge, there are at least three approaches that have been considered in the context of the POD method: the global reduced model (GRM), the local reduced model (LRM), and the adaptive reduced model (ARM). The GRM approximates the solution of interest in a subspace spanned by global basis vectors \cite{Schmidt:04a}. This method  can be straightforwardly applied to  a wide range of problems; however, in order to obtain a high accuracy, a subspace with relatively high dimension should be used to construct the reduced model, especially when  many solution modes exist for the whole interested domain. Thus, the GRM inevitably keeps some redundant dimensions in the online computation  and can lead to long online simulation times. In order to improve the computational efficiency, the LRM projects the original equation onto a subspace which corresponds to snapshots in one subdomain. All the precomputed snapshots are clustered either through time domain partitions \cite{DihlmannM:11a, Mohseni:14a}, space domain partitions \cite{FarhatC:12a, Willcox:13a}, or parameter domain partitions \cite{PateraAT:10a,  HaasdonkB:11a, Eftang:12a, MadayY:12a}.
In the LRM, the selected snapshots contribute equally to form the local POD subspace, while snapshots outside one subdomain are neglected. To overcome this, some ARMs use global data, but form  adaptive reduced bases  through  subspace interpolation methods, such as the angle interpolation \cite{LieuL:04a}, and the geometric interpolation in the Grassmann  manifold\cite{FarhatC:08a, FarhatC:09a}. It should be mentioned that the interpolation-based model reduction has been successfully applied in many areas of computational engineering including frequency response analysis \cite{Farhat:07b, Farhat:12b, Farhat:13b}, structural vibrations \cite{FarhatC:09a, HanS:02a}, and aeroelasticity \cite{FarhatC:08a, FarhatC:06a,FarhatC:10a, FarhatC:11c}. These  methods could effectively construct a new subspace from precomputed subspaces for each new parameter; however, the constructed subspace must have the same dimension with the precomputed ones. Thus, there is no flexibility to change the new subspace dimension in order to balance the accuracy and the computational speed of reduced models.

In this paper, we present a new adaptive reduced modeling technique via an adaptive weighting of the snapshots in the context of localized reduced bases; the dimension of our reduced model can be adaptively chosen to obtain certain desired levels of accuracy.  In our approach, the basis vectors are directly computed through the singular value decomposition (SVD) of a weighted snapshot matrix, where each column is defined as a  snapshot vector multiplied by a weighting coefficient. The method of determining weighting coefficients in the paper is similar to the method of determining interpolation coefficients  using a  kernel function. When a compact interpolation scheme is used, the subspace it generates is similar to the one from a local POD method. At the other extreme, when all the coefficients in the information matrix are set equally, this method degenerates to a global POD method.  Finally, when a non-compact scheme is used, such as radial basis function or inverse distance weighting, the subspace behaves similar to local POD when the dimension is low and similar to global POD when the dimension is high. We demonstrate that the reduced equation corresponding to this subspace gives much higher accuracy than the global reduced model with the same subspace dimension via numerical simulations. Furthermore, the proposed ARM could be formed by the discrete empirical interpolation method (DEIM) \cite{SorensenDC:10a}, and therefore  handle non-linearities that arise.

Another contribution of this paper is that the chord iteration  is introduced to model reduction to speed up the online computation of elliptic PDEs.  For elliptic PDEs, most existing model reduction techniques are focused on simplify the Newton iteration \cite{FarhatC:11a, FarhatC:11b, PateraAT:06a}. If the original system contains a nonlinear term, the classical DEIM framework \cite{SorensenDC:10a}  allows for a relatively inexpensive computation of a reduced Jacobian operator. However, there exists some computational redundancy to update the reduced Jacobian at each iteration. On one hand, computing a Jacobian matrix is usually more expensive than computing a vector field; this statement holds for both elliptic PDEs and their reduced versions, since evaluation of a  Jacobian matrix is based on a series evaluations of reduced vector fields in general. On the other hand, since  a reduced Jacobian based on the DEIM method is only an approximation of the original Jacobian, the reduced Newton iteration can only achieve a linear convergence rate, rather than a quadratic convergence rate in the standard Newton iteration. Motivated by this fact, we can save additional online computation time by approximating a reduced Jacobian during the offline stage. This is achieved by utilizing the chord iteration in the framework of the ARM.

The remainder of this article is organized as follows. In section \ref{sec:overview}, an overview of model reduction for parameterized PDEs and a general error analysis are presented. Section \ref{sec:POD} briefly reviews the classical POD-Galerkin method and its variant, the POD-DEIM. Then a general method for elliptic PDEs is discussed in section \ref{sec:elliptic}. After introducing the chord method, we use it in conjunction with the classical POD method for model reduction. Then adaptive reduced basis is proposed in order to use fewer modes to approximate the original system.  Section \ref{sec:parabolic} extends the ARM to parabolic PDEs. Finally, conclusions are offered in section \ref{sec:conclusion}.

\section{Formulation of Parameterized PDEs}\label{sec:overview}
We consider both parabolic PDEs and elliptic PDEs in this section. By discretization (for example, using finite difference or finite element methods), an elliptic parameterized PDE for variable $u\in \mathbb{R}^n$ with input parameter $\mu\in \mathbb{R}^d$ can be expressed as an algebraic equation
\begin{equation} \label{elpde}
f(\mu,u) = 0,
\end{equation}
where $f: \mathcal{D}\times \mathcal{R}\to \mathbb{R}^n$ is a smooth function for $\mathcal{D} \subset \mathbb{R}^d$ and $\mathcal{R} \subset \mathbb{R}^n$. For any fixed input parameter $\mu \in \mathcal{D}$, we seek a solution $u=u(\mu)\in \mathcal{R}$, such that $f(\mu,u)=0$ could be satisfied.

Parabolic PDEs are often used to describe dynamical systems with time dependent solutions. Let $\mathcal{I}=[0, T]$ denote the time domain and $\mathcal{D}$ denote the parameter domain. By spatial
discretization, the original parabolic PDE becomes an ordinary differential equation (ODE)
\begin{equation} \label{ode}
\dot u = f(t, \mu, u),
\end{equation}
with an initial condition $u(0,\mu)=u_0$, where $f: \mathcal{I} \times \mathcal{D} \times \mathcal{R} \to \mathbb{R}^n$ denotes the discretized vector field. For any fixed $t\in \mathcal{I}$ and $\mu\in \mathcal{D}$, the state variable $u=u(t,\mu)\in \mathcal{R}\subset \mathbb{R}^n$ satisfies (\ref{ode}). By definition, $u(t, \mu)$ is a flow which gives an orbit in $\mathbb{R}^n$ as $t$ varies over $\mathcal{I}$ for a fixed initial condition $u_0$ and a fixed input parameter $\mu\in \mathcal{D}$. The orbit contains a sequence of states (or state vectors) that follow from $u_0$.

In this article we shall respectively refer (\ref{elpde}) and (\ref{ode}) as discretized elliptic and parabolic PDEs, although they can represent more general discretized PDEs with parameter variation. In order to use the same framework to study (\ref{elpde}) and (\ref{ode}), we use $\tau$ to represent $\mu$ in (\ref{elpde}) and to represent $(t,\mu)$ in (\ref{ode}). Let $\mathcal{T}$ denote the input space, and $\mathcal{T}=\mathcal{D}$ or $\mathcal{T}=\mathcal{I} \times \mathcal{D}$. Then for both scenarios, $u(\tau)$ and $f(\tau, u)$ could be used to represent the solution snapshot and the vector field corresponding to $\tau \in \mathcal{T}$. It follows that (\ref{elpde}) and (\ref{ode}) become $f(\tau,u)=0$ and $\dot u=f(\tau, u)$ respectively.

An offline-online splitting scheme is used for the model reduction in this article. Let $N$ denote the ensemble size. In the offline stage, $\{\tau_i\}_{i=1}^N$ are sampled in the parameter space,and the corresponding solutions with their derivatives could induce a subspace, where the real solution approximately resides. Induced subspaces in the context of model reduction include the Lagrange, Taylor, and Hermite subspaces \cite{PorschingTA:85a}. In this article, we are focusing on constructing a reduced model in the Lagrange subspace $\mathcal{S}_r=\textnormal{span}\{u_i\}_{i=1}^N \subset \mathbb{R}^n$. The subspace dimension $r$ satisfies $r\le\min \{ n, N\}$. Let $\Phi_r:=[\varphi_1, \ldots, \varphi_r] $ contains an orthonormal basis $\{\varphi_i\}_{i=1}^r$ of $\mathcal{S}_r$. Let superscript $T$ denote the matrix transpose, and $I_r$ is the $r \times r$ identity matrix. Then $\Phi_r\in \mathcal{V}_{n,r}$, where $\mathcal{V}_{n,r}=\{A\in \mathbb{R}^{n\times r}|A^TA=I_r\}$ is denoted as the Stiefel manifold of orthonormal $r$-frames in $\mathbb{R}^n$. When $r\approx n$, a reduced equation constructed in $\mathcal{S}_r$ could not obtain significant speedups. One often seeks a $k(\ll r)$-dimensional linear subspace $\mathcal{S}_k\subset\mathcal{S}_r$ where most solution vectors approximately reside. Moreover, there exists an $n \times k$ orthonormal matrix $\Phi_k =[\phi_1, \ldots, \phi_k]$ whose column space is $\mathcal{S}_k$. Once the subspace is specified, a reduced model can be constructed by several approaches, such as Galerkin projection~\cite{KunischK:02a}, Petrov-Galerkin projection~\cite{FarhatC:11a}, symplectic Galerkin projection~\cite{mohseni:14b}, and empirical interpolation~\cite{PateraAT:04a,PateraAT:06a}.

Regardless of the techniques applied for model reduction, a key consideration is how to approximate the original system with high accuracy. The projection of a state variable $u \in \mathbb{R}^n$ onto $\mathcal{S}_r$ and $\mathcal{S}_k$ can be respectively presented by $\tilde u_r:=\Phi_r \Phi_r^T u$, and $\tilde u_k:=\Phi_k \Phi_k^T u$. Let $e_r:=u-\tilde u_r$ denote the difference between a solution vector $u$ and its projection on $\mathcal{S}_r$, and $e_k:=u-\tilde u_k$ denote the difference between $u$ and its projection on $\mathcal{S}_k$. In addition, we define $e_o:=\tilde u_r-\tilde u_k$ as the difference between these two projections of $u$.

Suppose the reduced model has a unique solution, $\hat u$, and $\hat u=\hat u(\tau)\in\mathcal{S}_k$. Usually, $\tilde u_k \ne \hat u$, and we use $e_i:=\tilde u_k-\hat u$ to represent their difference. Moreover, numerical simulation inevitably introduces further the numerical error $e_t$, such as discretization of time integration and the round-off error. This kind of error exists for both high dimensional and low dimensional simulations. However, for simplicity we assume the solution to the reduced model $\hat u$ is obtainable by an accurate numerical scheme, and neglect $e_t$. The total error $e$ of the approximating solution $\hat u$ from a reduced equation, can be decomposed into three components: $e = {e_r}+{e_o} + {e_i}$, and these components are orthogonal to each other.

Decreasing the magnitude of the projection error $e_k(=e_r+e_o)$ is the key to decrease the total error. On one hand, $e_k$ provides a lower bound for the reduced model, as $\left\| {{e_k}} \right\| \le \left\| e \right\|$ is always satisfied. On the other hand, for both elliptic \cite{PorschingTA:85a} and parabolic PDEs \cite{PetzoldLR:03a}, if the Galerkin method is used to produce the reduced equation, there respectively exist a constant $C$ such that $\left\| e \right\| \le C\left\| {{e_k}} \right\|$. Therefore, a upper error bound of $e$ is also related to $e_k$. The first component $e_r$ of $e_k$ is directly related to sampling input parameters during the offline stage. One could use uniform sampling process in the parameter space, or use nonuniform sampling process through a greedy algorithm \cite{PateraAT:04a,PateraAT:06a}. The second component $e_o$ of $e_k$ comes from the error of dimensionality reduction. If the global POD method is used, then $e_o$ is related to the truncation of the SVD. In this article, we are about to discuss an adaptive method to form a reduced subspace $\mathcal{S}_k$ such that $\left\| {{e_o}} \right\|$ could reach a lower value with fixed $k$. We will begin with a brief review of POD, which paves a way to introduce the ARM.

\section{Proper Orthogonal Decomposition (POD)}\label{sec:POD}
 In this section, we give a brief overview of the procedure and properties of POD. In a finite dimensional space, it is essentially the same as the singular value decomposition (SVD). Let $X=[u_1, \ldots, u_N]$ be a $n\times N$ snapshot matrix, where each column $u_i=u(\tau_i)$ represents a solution snapshot corresponding to input parameters $\tau_i$. The POD method constructs a basis matrix $\Phi_k$ that solves the following minimization problem
\begin{equation}\label{opt}
\mathop {\min }\limits_{{\Phi _k} \in {\mathcal{V}_{n,k}}}
{\left\| {(I - {\Phi _k}{\Phi _k}^T)X} \right\|_F}.
\end{equation}
Thus, the basis matrix $\Phi_k$ minimizes the Frobenius norm of the difference between $X$ with its projection $\tilde X:=\Phi_k\Phi_k^TX$ onto $\mathcal{S}_k$. Since the dimension of the Lagrange subspace $\mathcal{S}$ is $r$, it follows that ${\rm{rank}}({\rm{X}}) = r$. Thus, the SVD of $X$ gives
\begin{equation}
X = V \Lambda W^T,
\end{equation}
where $V\in \mathcal{V}_{n,r}$, $W\in \mathcal{V}_{p,r}$, and $\Lambda={\rm{diag}}(\lambda_1, \ldots, \lambda_r)\in \mathbb{R}^{r\times r} $ with $\lambda _1 \ge \lambda _2 \ge \ldots \ge \lambda _r > 0$. The $\lambda$s are called the singular values of $X$. In many applications, the truncated SVD is more economical, where only the first $k$ columns $V'$ of $V$ and the first $k$ columns $W'$ of $W$ corresponding to the diagonal matrix $\Lambda'$ corresponding to the $k$ largest singular values are calculated, and the rest of the matrices are not computed. Then the projection of $X$ is given by
\begin{equation}
\tilde X=V'\Lambda' W'^T,
\end{equation}
and the solution of $\Phi_k$ in (\ref{opt}) is given by $\Phi_k=V'$. Moreover, the projection error of (\ref{opt}) in Frobenius norm by the POD method is given by
\begin{equation}\label{podtruc}
E_k={\left\| {(I - {\Phi _k}{\Phi _k}^T)X} \right\|_F} = \sqrt
{\sum\limits_{i = k + 1}^r {\lambda _i^2} }.
\end{equation}

The key notion of POD and other projection-based reduced models is to find a $k$-dimensional subspace $\mathcal{S}_k$ on which all the state vectors live. Although the truncated SVD is no longer an exact decomposition of the original matrix $X$, it provides the best approximation $\tilde X$ to the original data $X$ with least Frobenius norm under the constraint that $\dim(\tilde X)=k$. In the rest of this paper, the SVD refers to the truncated SVD unless specified otherwise.

\subsection{Galerkin Projection}
Let $v_k \in \mathbb{R}^k$ denote the state variable in the subspace coordinate system, and $\hat u_k=\Phi_k v_k$ denote the same state in the original coordinate system. Projecting the system (\ref{elpde}) onto $\mathcal{S}_k$, one obtains the reduced model of an elliptic PDE,
\begin{equation}\label{romelp}
\Phi_k^T f(\tau, \Phi_k v_k)=0,
\end{equation}
where $\tau=\mu\in \mathcal{D}$ denotes the input parameter. Analogously, a reduced model of a parabolic PDE can be obtained by projecting the system (\ref{ode}) on to $\mathcal{S}_k$,
\begin{equation}\label{rompara}
\dot v_k= \Phi_k^T f(\tau, \Phi_k v_k),
\end{equation}
and $\tau=(t, \mu)\in \mathcal{I}\times \mathcal{D}$ is used to identify the solution trajectory corresponding to $\mu$ at time $t$.

\subsection{Discrete Empirical Interpolation Method (DEIM)}\label{sec:DEIM}

Equations (\ref{romelp}) and (\ref{rompara}) are reduced equations formed by the Galerkin projection. In fact, they can achieve fast computation only when the analytical formula of the reduced vector field $\Phi^T f(\tau,\Phi_k v_k)$ can be significantly simplified, especially when it is a linear (or a polynomial) function of $v_k$. Otherwise, one will need to compute the state variable in the original system $\Phi_k v_k$, evaluate the nonlinear vector field $f$ at each element, and then project $f$ onto $\mathcal{S}_k$. In this case, the reduced models (\ref{romelp}) and (\ref{rompara}) are more expensive than the correspondingly full models.
In recent years,  many variants of POD-Galerkin   were
developed to reduce the complexity of evaluating the nonlinear
term of vector field, such as   trajectory piecewise
linear and quadratic approximations~\cite{WhiteJ:03a, WhiteJ:06a,
WeileDS:01, WhiteJ:00a}, missing
point estimation~\cite{AstridP:04a,WillcoxK:08a}, Gappy POD method~\cite{EversonR:95a, WillcoxK:04a, WillcoxK:06a, FarhatC:11a, Farhat:13a}, empirical
interpolation   method~\cite{PateraAT:04a,PateraAT:06a},  and DEIM~\cite{SorensenDC:10a, DrohmannM:12a}.
 Since our numerical simulation applies the DEIM, we briefly review this method in this section.

The original vector field, $f(\tau, u)$ can be split into a linear part and a nonlinear part, i.e.,
\begin{equation}\label{decompln}
 f(\tau, u)=L u+ f_N(\tau,u).
 \end{equation}
where $L\in \mathbb{R}^{n\times n}$ is a linear operator and $f_N(\tau, u)$ denotes the nonlinear vector term. Using the Galerkin projection, the reduced vector field is given by
\begin{equation} \label{decomp}
\Phi_k^T f(\tau, \Phi_k v_k) = \hat L v_k + \Phi_k^T f_N (\tau,
\Phi_k v_k),
\end{equation}
where $\hat L=\Phi_k^T L \Phi_k\in \mathbb{R}^{k\times k}$. Unless the nonlinear term can be analytically simplified, its computational complexity still depends on $n$. An effective way to overcome this difficulty is to compute the nonlinear term at a small number of points and estimate its value at all the other points. Considering $u$ is a smooth function of $\tau$, we can define a \emph{nonlinear snapshot} $ g(\tau) =:f_N(\tau, u(\tau))$ for $\tau \in \mathcal{T}$. Then the reduced vector field from (\ref{decompln}) restricted on $ \tau \in \mathcal{T}$ and $u=u(\tau)$ can be approximated as
\begin{equation} \label{deimapprox}
\Phi_k^T f(\tau, u(\tau)) = \hat L v_k + [\Phi_k^T \Psi_m(P^T
\Psi_m)^{-1}] [P^T g (\tau)],
\end{equation}
where $\Psi_m$ is an $n\times m$ matrix that denotes the collateral POD basis based on a precomputed nonlinear snapshot ensemble, and $P^T$ is an $m \times n$ index matrix to project a vector of dimension $n$ onto its $m$ elements. For example, if $g=[g_1; \ldots; g_4]$, and $P^T=[1 \ 0 \ 0 \ 0; \ 0 \ 0 \ 1\ 0]$, then $P^Tg=[g_1; g_3]$. $P$ can be computed by an offline greedy algorithm. We recommend readers to refer \cite{SorensenDC:10a} for more details. Notice that $\Phi_k^T \Psi_m(P^T \Psi_m)^{-1}$ is calculated only once at the outset and $P^T g (\tau)$ is only evaluated on $m$ elements of $g(\tau)$, therefore it is very efficient when $m\ll n$. Using the POD-DEIM approach, we will respectively study nonlinear elliptic and parabolic PDEs in the next two sections.

\section{Nonlinear Elliptic PDEs}\label{sec:elliptic}
We are focusing on model reduction for  nonlinear elliptic PDEs in this section. The algorithm in this section could be considered as a reduced-order extension for the chord method. We follow the offline-online splitting computational strategy, and use the DEIM to treat nonlinear terms in the original system.

The general form of elliptic PDEs, after discretization, is given by a nonlinear algebraic equations (\ref{elpde}). Let $N$ denote the ensemble size. In the offline stage, $\{\mu_i\}_{i=1}^N$ are sampled in the parameter space, and we solve the corresponding solutions $\{u_i\}_{i=1}^N$. If $u_i=u(\mu_i)$ satisfies (\ref{elpde}), and the Jacobian matrix $J_i:=D_u f(\mu_i, u_i)$ is nonsingular, then by the implicit function theorem, there exists neighborhoods $V_i\subset \mathbb{R}^m$ of $\mu_i$, and $U_i\subset \mathbb{R}^n$ of $u_i$ and a smooth map $u:V_i\to U_i$ such that locally, (\ref{elpde}) has a unique solution. If a new input parameter ${\mu _*} \in { \cup _i}{V_i}$, then the following equation has a unique solution,
\begin{equation}\label{hd}
 F(u)=0,
\end{equation}
where $F(u):=f(\mu_*,u)$.

\subsection{Adaptive Reduced Model (ARM)}\label{sec:ARM}
In order to solve (\ref{hd}) via a reduced model, we consider three SVD-based approaches to construct a subspace $\mathcal{S}_k$ from precomputed snapshots. The first approach is the standard POD method, which constructs a \emph{global reduced model} (GRM) from all the snapshots in the ensemble. Defining a matrix of $N$ snapshots
\begin{equation}\label{snapX}
X = [u_1, \ldots, u_N],
\end{equation}
then the POD basis matrix $\Phi_k$ is constructed from SVD of $X$. The projection error of $X$ in Frobenius norm is given by (\ref{podtruc}). If the problem depends on many parameters or if the solution shows a high variability with the parameters, a relatively high dimensional reduced space is needed in order to represent all possible solution variations well. This effect is even considerably increased when treating evolution problems with significant solution variations in time. Another aspect is the fact that projection-based model reduction techniques, such as POD, usually generate small but full matrices while common discretization techniques (such as the finite difference method) could lead to large but sparse matrices. Unless the reduced model has significantly lower dimension, it is even possible that the reduced model is more time consuming to evaluate than the original model.

The second approach is to construct a \emph{local reduced model} (LRM), which partitions the interested parameter domain $\mathcal{D}$ into some disjoint subdomains $\mathcal{D}_i$ and forms a local snapshot matrix for each subdomain.    Let ${\left\| {{\mu _i} - {\mu _j}}\right\|}$ denotes the distance between $\mu _i$ and $\mu _j$ in the parameter domain. Without additional information about parameter domain,  the Euclidean norm can be used here for simplicity. Suppose $\mu_i$ is the nearest neighbor of $\mu_*$ among $\{\mu_i\}_{i=1}^N$, then  $\mu_*\in \mathcal{D}_i$.  Although we do not need to explicitly compute the domain partition, this partition achieves the same effect as the Voronoi diagram, where each $\mu_i$ is a reference point of $\mathcal{D}_i$. In order to describe the local neighborhood relationships between precomputed data points, one can construct the \emph{$\varepsilon$-neighborhood graph} or the \emph{$k$-nearest neighbor graph} for the vertices $\{\mu_i\}_{i=1}^N$. In the first case,  we connect all vertices whose pairwise distances are smaller than $\varepsilon$.
In the second case,  $\mu_i$ and $\mu_j$ are connected with an edge if $\mu_i$ is among the $k$-nearest neighbors of $\mu_j$ or if $\mu_j$ is among the $k$-nearest neighbors of $\mu_i$. Let $l$ input parameters  $\{\mu_{j_1}, \ldots, \mu_{j_l}\}$ be neighbors of $\mu_i$, then a local snapshot matrix is defined by
\begin{equation}\label{snaploc}
X_i^L = [u_{j_1}, \ldots, u_{j_l}].
\end{equation}
  Let $u_*$ be the solution corresponds to the input parameter $\mu_*$, i.e., $f(\mu_*,u_*)=0$. If $u_*$ approximately resides on a subspace spanned by the neighbours of $u_i$, the subspace can be constructed by the SVD of $X_i^L$.
Since $X_i^L$ contains fewer snapshots than $X$, it is expected that the LRM needs lower dimension to approximate the original system. Equivalently, if the LRM has the same dimension as the GRM, it has less truncation error of the SVD. It should be mentioned that if $N$ is large, one can pick a few $\mu_i$s as reference points to construct a smaller number of subdomains.

The LRM is usually referred as local PCA in many fields of computer science including web-searching, information retrieval, data mining, pattern recognition and computer vision. Moreover, the idea of local neighbourhood graphs mentioned above is  also widely used in   other main techniques for the nonlinear dimensionality reduction, such as locally linear embedding \cite{RoweisST:00a},  Laplacian eigenmaps \cite{BelkinM:02a}, and Isomap \cite{TenenbaumJB:00a}. All these techniques can  successfully discover the locally linear structure  when  there are a large number of vertices in each neighbourhood. However, for the model reduction of PDEs, it is usually very expensive to obtain many solution snapshots, as they require solving the original problems during the offline stage. Without a large number of data points for each neighbourhood, the LRM as well as other nonlinear dimensionality reduction techniques may not be able to yield  accurate solutions without  taking advantage the information from some other partitioned subdomains.

We extend the idea of the LRM and the  fully connected graph to propose the third approach for model reduction, the \emph{adaptive reduced model} (ARM), to compute the adaptive reduced bases for parameter variation.  Here  all pairwise points are connected with a weighting matrix. As the graph should emphasize the local neighborhood relationships, the element  $a_{ij}$ of the weighting matrix has a large value when $\mu_i$ and $\mu_j$ are close.
An example for such a weighting function is the Gaussian  function
\begin{equation}\label{weight}
a_{ij} = {{\exp ( { - {{{{\left\| {{\mu _i} - {\mu _j}}\right\|}^2}}}/{{2{\sigma ^2}}}} )}},
\end{equation} where $\sigma$ controls  the kernel length.
Thus, $a_{ij}$ and satisfies $0< a_{ij}\le 1$.

Suppose $\mu_*\in \mathcal{D}_i$. The direct projection  of $u_*$ onto a subspace spanned by $\Phi_r$  can be written as a linear combination of $a_{ij}u_j$,
\begin{equation}\label{ustar}
 \tilde u_r^A(\mu_*) = \sum\limits_{j = 1}^N
{{\eta _j} {a_{ij}}{u_j}} ,
\end{equation}
where $\eta_j$ is the coefficient, and  superscript $A$ is used to denote the proposed ARM. A weighted snapshot matrix for the $i$th subdomain can be defined as
\begin{equation}\label{info}
X^A_i = [{a_{i1}}{u_1},\ldots,{a_{iN}}{u_N}],
\end{equation}
where $0\le a_{ij}\le 1$ for each $a_{ij}$.

 When rank$(X^A_i)>k$, the SVD can be used to extract the first $k$ dominant modes from $X^A_i$, and obtain a lower dimensional subspace. Especially, the ARM degenerates to the GRM if $\sigma \to \infty$ and $a_{ij}=1$ for each $i,j$.
  The Gaussian weighting function can also be replaced by a compact weighting function. Then, the ARM degenerates to the LRM if $a_{ij}=1$ for $\|\mu_i-\mu_j\|<\epsilon$ and $a_{ij}=0$ otherwise. For convenience, we remove the subscript $i$ hereafter. Let the column vectors of $\Phi_k^A \in \mathcal{V}_{n \times k}$ span the POD subspace of $X^A$. As an analogy of $E_k$ in (\ref{podtruc}), the projection error of $X^A$ onto $\mathcal{S}_k^A$ in Frobenius norm is given by
\begin{equation}\label{trcerror}
E_k^A={\left\| {(I_k - { \Phi _k^A}{( \Phi _k^A)}^T)X} \right\|_F}
= \sqrt {\sum\limits_{j = k + 1}^r {(\lambda _j^A)^2} },
\end{equation}
where $\lambda _j^A$ is the $j$th singular value of $X^A$.

When rank$(X^A) \le k$, the SVD or the Gram-Schmidt process could be used to obtain $r$ orthonormal basis vectors that span $\mathcal{S}_r^A$. Choose any additional $k-r$ vectors to form an $n\times k$ matrix $\Phi_k^A \in \mathcal{V}_{n \times k}$, (\ref{trcerror}) becomes $E^A_k=0$. Comparing $E_k$ with $E_k^A$, we have the following lemma.

\medskip
\begin{lemma} \label{lem:trunc}
Let $X$ and $X^A$ be the matrices for solution snapshots and weighted solution snapshots. $E_k$ and $E_k^A$ are projection errors respectively given by (\ref{podtruc}) and (\ref{trcerror}). Then, $E_k \ge E_k^A$.
\end{lemma}

\begin{proof}
When rank$(X^A) \le k$, the conclusion holds trivially. We consider the case for $r^A>k$. For each weighing coefficient $a_j$ in $X^A$, we have $0\le a_j\le 1$. It follows that
\begin{equation}
\begin{array}{l}
 {\left\| {(I_k - {\Phi _k}{\Phi _k}^T)X} \right\|_F} = {\left\| {(I_k - {\Phi _k}{\Phi _k}^T){u_1},\ldots,(I_k - {\Phi _k}{\Phi _k}^T){u_N}} \right\|_F} \\
 \ \ \ \ \ \ \ \ \ \ \ \ \ \ \ \ \ \ \ \ \ \ \ \ \ \ \ \ge {\left\| {(I_k - {\Phi _k}{\Phi _k}^T){a_1}{u_1},\ldots,(I_k - {\Phi _k}{\Phi _k}^T){a_N}{u_N}} \right\|_F} \\
 \ \ \ \ \ \ \ \ \ \ \ \ \ \ \ \ \ \ \ \ \ \ \ \ \ \ \ = {\left\| {(I_k - {\Phi _k}{\Phi _k}^T) X^A} \right\|_F} \\
  \ \ \ \ \ \ \ \ \ \ \ \ \ \ \ \ \ \ \ \ \ \ \ \ \ \ \ \ge {\left\| {(I_k - {{ \Phi }_k^A}({ \Phi }_k^A)^T) X^A} \right\|_F}. \\
 \end{array}
 \end{equation}
The last inequity holds because $\Phi _k^A$ provides a least Frobenius norm for the difference of matrix $X^A$ and its projection onto a $k$-dimensional subspace. Using the definition of $E_k$ and $E_k^A$, one obtains $E_k \ge E_k^A$. Especially, $E_k = E_k^A$ holds if and only if $a_i=1$ for any $i$. In this case, the ARM degenerates to the GRM.
 \hfill
\end{proof}

\medskip
\begin{proposition}\label{cor:trunc}
Let $X$ and $X^L$ be the matrices for solution snapshots and local solution snapshots. $E_k$ and $E_k^L$ denote the corresponding projection errors. Then, $E_k > E_k^L$.
\end{proposition}
\begin{proof}
We construct an $n\times N$ matrix $X^L_{ext}=[b_1 u_1, \ldots, b_N u_N]$. Let $b_j=1$ when $\mu_j$ is in the $i$th subdomain and $b_j=0$ otherwise. Essentially, $X^L_{ext}$ is an extension of $X^L$ with some 0 column vectors. The SVD of $X^L_{ext}$ and $X^L$ gives the same POD basis matrix $\Phi^L_k$ and singular values. Therefore, the projection error in Frobenius norm, $E_k^L$, is given by $\left\|{(I_k - {\Phi _k^L}({\Phi _k^L}){^T}) X^L_{ext}} \right\|_F$. Since $b_j=0$ for some $j$, by lemma \ref{lem:trunc}, it follows that
\begin{equation}
{\left\| {(I_k - {\Phi _k}{\Phi _k}^T)X} \right\|_F} >
\left\|{(I_k - {\Phi _k^L}({\Phi _k^L}){^T}) X^L_{ext}} \right\|_F.
 \end{equation}
The above equation means $E_k > E_k^L$, i.e., the projection error of the LRM is smaller than the error from the GRM.
 \hfill
\end{proof}

%Lemma \ref{lem:orthogonal} indicates that the truncation error of
%SVD is a lower error bound of the total error.
\medskip

Lemma \ref{lem:trunc} and proposition \ref{cor:trunc} compare projection error of the same data ensemble $\{u_i\}_{i=1}^N$, and suggests that $E^A_k$ and $E^L_k$ is usually smaller than $E_k$. We choose a constant $\epsilon_\eta$ such that $|\eta _j| < \epsilon_\eta$ for each $j$. For a specified parameter $\mu_*$, if the projection ${\tilde u_r^A}$ of $u_*$ has a form of (\ref{ustar}), the SVD truncation error $e_o^A$ of the ARM is bounded by a constant times $E_k^A$,
\begin{equation}\label{ARMerr}
\begin{array}{l}
 \left\| {{e_o^A}} \right\| = {\left\| {{{\tilde u}_r^A} - \tilde u_k^A} \right\|} = {\left\| {\sum\limits_{j = 1}^N {{\eta _j}{a_j}{u_j} - {\eta _j}{a_j}(\Phi _k^A{{(\Phi _k^A)}^T}){u_j}} } \right\|} \\
 \ \ \ \ \ \ \ \ \le \sum\limits_{j = 1}^N {{{\left\| {(I - \Phi _k^A{{(\Phi _k^A)}^T}){\eta _j}{a_j}{u_j}} \right\|}}} \le
 \epsilon_\eta \sum\limits_{j = 1}^N {{{\left\| {(I - \Phi _k^A{{(\Phi _k^A)}^T}){a_j}{u_j}} \right\|}}} = \epsilon_\eta {E_k^A}. \\
 \end{array}
 \end{equation}
The error bound of $e_o$ of the GRM has the similar property
\begin{equation}\label{GRMerr}
\begin{array}{l}
 \left\| {{e_o}} \right\| = {\left\| {{{\tilde u}_r} - {{\tilde u}_k}} \right\|} = {\left\| {\sum\limits_{j = 1}^N { {\eta _j}{a_j}{u_j} - {\eta _j}{a_j}({\Phi _k}\Phi _k^T){u_j}} } \right\|} \\
 \ \ \ \ \ \ \ \ \le \sum\limits_{j = 1}^N {{{\left\| {(I - {\Phi _k}\Phi _k^T){\eta _j}{a_j}{u_j}} \right\|}}} \le \mathop {\max }\limits_j (|\eta _j a_j|)\sum\limits_{j = 1}^N {{{\left\| {(I - {\Phi _k}\Phi _k^T){u_j}} \right\|}}} \le  {\epsilon _\eta }{{E_k}} \\
 \end{array}
\end{equation}
Combined (\ref{ARMerr}) with (\ref{GRMerr}), and using lemma (\ref{lem:trunc}), we conclude that the upper error bound of $e_o^A$ is smaller than $e_o$.

We next consider the projection error $e_r$ of $u_*$. Since the GRM uses all the solution snapshots to form a snapshot matrix, it immediately follows that $\|e_r\|\le \|e_r^A\|$, and $\|e_r\|<\|e_r^L\|$. Specifically, if a noncompact scheme is used for the ARM, we have $\|e_r^A\|=\|e_r\|$.

\subsection{Chord Iteration}\label{sec:offline}
The Newton iteration and its reduced version are widely used to solve elliptic PDEs \cite{FarhatC:11a, FarhatC:11b, PateraAT:06a}. If $F$ in Equation (\ref{hd}) contains a nonlinear term, the DEIM method \cite{SorensenDC:10a} allows for a relatively inexpensive approximation of a reduced Jacobian operator. In the online computation, the most expensive procedures of the reduced Newton iteration are to compute the $k\times k$ reduced Jacobian matrix $\hat J$ for each iteration, since evaluation of $\hat J$ is based on a series evaluations of reduced vector fields. However, since the DEIM method only provide an approximation for $\hat J$, there is no need to update $\hat J$ for each iteration. Motivated by this fact, we can apply model reduction techniques to simplify the chord iteration so that it can solve a general elliptic PDE with an additional computational saving.

The original chord iteration computes $J_0=F'(u(0))$ at the outset, and use $J_0$ to approximate the Jacobian for each iteration. Specifically, for iteration $j$, we first compute the vector $F(u(j))$. Then, solve
\begin{equation}\label{chord1}
J_0\xi (j)=-F(u(j))
\end{equation}
for $\xi (j)$. After that, update the approximating solution,
\begin{equation}\label{chord2}
u(j+1)=u(j)+\xi(j). \end{equation}

Under certain conditions, the chord iteration could obtain a convergent solution if the initial trial solution is close enough to the actual solution, as given by the following lemmas \cite{Kelley:95a}.

\medskip
\begin{lemma} \label{lem:newton}
Suppose (\ref{hd}) has a solution $u_*$, $F'$ is Lipschitz continuous with Lipschitz constant $\gamma$, and $F'(u_*)$ is nonsingular. Then there are $\bar K>0$, $\delta >0$, and $\delta_1$ such that if $u(j)\in \mathcal{B}_{u_*}(\delta)$ and $\|\Delta(u(j))\|<\delta _1$ then
\begin{equation}\label{generalerror}
u(j+1)=u(j)-(F'(u(j))+\Delta(u(j)))^{-1}(F(u(j))+\epsilon(u(j)) )
\end{equation}
is well-defined and satisfies
\begin{equation}
\|e(j+1)\|\le \bar K(\|e(j)\|^2+\|\Delta (u(j))\|
\|e(j)\|+\|\epsilon(u(j))\|),
\end{equation}
where $e(j):=u_*-u(j)$ denotes the error for iteration $j$.
\end{lemma}

For chord iteration, $\epsilon (u(j))=0$, $\Delta(u(j))=F'(u_0)-F'(u(j))$. If $u_0, u(j) \in \mathcal{B}_{u_*}(\delta)$, $\|\Delta(u(j))\| \le \gamma \|u_0-u(j)\|\le \gamma (\|e(0)\|+\|e(j)\|)$. Using lemma \ref{lem:newton}, the following lemma is obtained, where $K_C:=\bar K(1+2 \gamma)$.

\medskip
\begin{lemma} \label{lem:chord}
Let the assumptions of lemma \ref{lem:newton} hold. Then there are $K_C>0$ and $\delta>0$ such that if $u_0\in \mathcal{B}_{u_*}(\delta)$ the chord iteration converges linearly to $u_*$ and
\begin{equation}
\|e(j+1)\|\le K_C \|e(0)\| \|e(j)\|.
\end{equation}
\end{lemma}
We suggest readers to refer to \cite{Kelley:95a} for proofs and more details. In the chord method, the complexity of (\ref{chord1}), (\ref{chord2}), and direct evaluation of $F(u)$ depends on $n$. If the dimension $n$ in (\ref{hd}) is very large, the chord method could be still prohibitively expensive for real time computation. For this reason, a model reduction approach is applied for decreasing the computational cost based on the chord method. A manifold learning procedure is used to extract the dominant modes from the original data. Usually, this procedure is very intensive in exchange for greatly decreased online cost for each new input-output evaluation. In the online stage, a reduced version of the chord method is used to solve for the solution. Since the solution can be solved in a low dimensional subspace, the complexity of online computation can be very low.

\subsection{Adaptive Reduced Model based on Chord Iteration}
In this subsection, the reduced chord iteration is combined with the ARM to solve an elliptic PDE. In the offline stage, for each input $\mu_i$, the solution snapshot $u_i$, the nonlinear snapshot $g_i$, and the corresponding Jacobian matrix $J_i$ are recorded to form an ensemble $\{\mu_i, u_i, g_i, J_i\}_{i=1}^N$. In the offline stage, we can use POD-Galerkin approach (for a linear PDE) or POD-DEIM (for a nonlinear PDE) to accelerate the computation.

As described in section \ref{sec:ARM}, for each input parameter $\mu_*$, one must first determine one subdomain that $\mu_*$ resides; we choose the subdomain $i$ such that $\|\mu_*-\mu_i\|$ get a minimal value. If  $\{\mu_i\}_{i=1}^N$ represents an  integer lattice in the parameter domain, then we can immediately find the $i$. Otherwise, searching the optimal $i$ is based on the data structure of the precomputed data ensemble. Usually this process is always affordable as long as  $\dim(\mu_*) \ll n$.

Then one can obtain the POD basis matrix $\Phi_k^A (\mu_i)$.
Let $v(j)\in \mathbb{R}^k$ be the reduced state at iteration $j$, $v(0)=v_{i}=(\Phi_k^A)^T u_{i}$ be the initial trial solution, $\hat J_i=(\Phi_k^A)^T J_i \Phi_k^A \in \mathbb{R}^{k\times k}$ be the reduced Jabobian. The Galerkin projection can be used to form reduced equations for (\ref{chord1}) and (\ref{chord2}) in chord method,
\begin{equation} \label{galerkinnewton}
\hat J_i \hat \xi (j)=-(\Phi_k^A)^T F(\Phi_k^A v(j)),
\end{equation}
\begin{equation}\label{redupdate}
 v(j+1)=v(j)+\hat \xi (j).
\end{equation}

As mentioned in the previous section, POD-Galerkin approach cannot effectively reduce the complexity for high dimensional systems when a general nonlinearity is present, since the cost of computing $(\Phi_k^A)^T F(\Phi_k^A v)$ depends on the dimension of the original system, $n$. In order to obtain significant speedups for a general nonlinear system, the DEIM can be used to approximate (\ref{redupdate}).

When $J_{i}$ is a symmetric positive definite (SPD) matrix, $\hat J_{i}$ is a nonsingular matrix and (\ref{galerkinnewton}) is well-defined \cite{SaadY:03a}. Moreover, it minimizes the mean square error in the search direction for each iteration. Unfortunately, the Jacobians of a nonlinear problem are not, in general, SPD matrices. If $\hat J_{i}$ is (near) singular, one can always choose another parameter nearby to obtain a nonsingular Jacobian. Otherwise, the original Jacobian is zeros almost everywhere in a neighbourhood of $(\mu_{i}, u_{i})$, which implies that the vector field of subdomain $i$ is a constant. Consequently, one can directly avoid online computation in this region.

Algorithm \ref{alg:redchord} lists all the procedures of the reduced chord iteration. An ensemble $\{\mu_i, u_i, g_i, J_i)\}_{i=1}^N$ is precomputed offline. The POD basis and the collateral POD basis are respectively computed in step 1. Some matrices involving the DEIM approximation are precomputed in step 2. Step 3 and step 4 involves computing the reduced Jacobian and the reduced initial state for each subdomain.
In the online stage, step 5 determines the subdomain $i$ where a new input parameter $\mu_*$ resides.  Step 6 simply picks up the trial solution and the estimated Jacobian computed in step 4 during the offline stage. This step does not involves any real computations. Step 5 and 6 are carried out only once. Steps 7-9 form the main loop for the online computation using the subspace coordinates and their complexity is independent of $n$. Therefore, the online computation of algorithm \ref{alg:redchord} is very efficient.

\begin{algorithm}
\caption{Reduced Chord method} \label{alg:redchord}
\begin{algorithmic}
 \REQUIRE
Precomputed ensemble $\{\mu_i, u_i, g_i, J_i)\}_{i=1}^N$.
\ENSURE Solution $u_*$ that satisfies $F(u)=0$. \STATE
\textbf{Offline:}
 \FOR {subdomain $i=1$ to $N$ }
\STATE 1: Use the SVD to compute the adaptive POD
basis $\Phi_k^A(\mu_{i})$ for the weighted solution matrix $X^A_i$, and the collateral
POD basis $\Psi_m^A(\mu_{i})$ for the weighted matrix for the nonlinear vector terms.
 \STATE 2: Use the DEIM to compute $\tilde L(\mu_{i})=(\Phi_k^A)^T L \Phi_k^A$ for the linear operator and $(\Phi_k^A)^T
\Psi_m^A (P^T\Psi_m^A)^{-1}$ for the nonlinear term.
 \STATE 3: Compute the reduced Jacobian $\hat J_{i}$. If it
 is nonsingular go back to step 1 and consider another point in subdomain $i$. If the new reduced Jacobian is still singular, label the subdomain $i$ as ``constant''.
 \STATE 4: Compute solution snapshots in the reduced coordinate system
$v_{i}=(\Phi_k^A)^T u_{i}$.
\ENDFOR
\STATE \textbf{Online:} \STATE 5: Choose the label $i\in\{1, \ldots, N\}$  such that the distance $\|\mu_*-\mu_{i}\|$ obtain the minimal value  and subdomain $i$ is not  labeled as ``constant''.
\STATE 6: Set
$v_{i}$ as the initially trial solution $v(0)$ in the reduced
coordinate system and set $\hat J_{i}$ as the approximated
Jacobian.
 \FOR {$j=0, \ldots,$ (until
convergence)}
 \STATE 7: Compute $(\Phi_k^A)^T \hat F(\Phi_k^A
 v(j))$, where $\hat F$ is an approximation of $F$ by DEIM.
 \STATE 8: Solve $\hat J_i \hat \xi (j)=-(\Phi_k^A)^T F(\Phi_k^A v(j))$,
 as (\ref{galerkinnewton}).
 \STATE 9: Update $v(j+1)=v(j)+\hat \xi(j)$, as
 (\ref{redupdate}).
\ENDFOR \STATE 10: Set $\hat u_*=\Phi_k^A v(j+1)$.
\end{algorithmic}
\end{algorithm}

For a fixed $\tau$ in (\ref{romelp}), the reduced algebraic equation formed by the Galerkin projection is given by
\begin{equation}\label{relpde}
(\Phi^A_k)^T F(\Phi^A_k v)=0.
\end{equation}
Usually, the reduced chord iteration could not converge to the actual solution, $v_*$, of (\ref{relpde}) since DEIM provides an extra error for the approximation for the vector field. However, suppose the DEIM approximation gives a uniform error bound, $\varepsilon_F$, for an interested domain of $u$, an error bound of algorithm \ref{alg:redchord} could be obtained in terms of $\varepsilon_F$. In the following lemma, we slightly abuse the notation and use $e$ to denote the error between $v_*$ and the approximating solution in reduced chord iteration. Meanwhile the superscript $A$ is removed for convenience.

\medskip
\begin{lemma} \label{lem:reducedchord}
Suppose (\ref{relpde}) has a solution $v_*$, $F'$ is Lipschitz continuous with Lipschitz constant $\gamma$, and $F'(\Phi_k v_*)$ is nonsingular. Suppose the DEIM approximation gives a uniform error bound $\|F(u)-\hat F(u)\|<\epsilon_F$ for any $u\in \mathcal{B}_{\Phi_k v_*} $. Let  $e(j):=v_*-v(j)$ denote the error for iteration $j$. Then there are $\bar K>0$, $K_C>0$, $\delta>0$ such that if $u_i\in \mathcal{B}_{\Phi_k v_*}(\delta)$, then the reduced chord iteration approach to $v_*$ with the error bound of $\|e(j)\|$ given by $\bar K{\varepsilon _F}/(1-K_C \delta )$ as $j \to \infty$.
\end{lemma}

\begin{proof}
In algorithm \ref{alg:redchord}, a sequence $\{v(j)\}$ is obtained by the following iteration rule,
\begin{equation}
v(j + 1) = v(j) - {\hat J_i^{ - 1}}(\Phi _k^T\hat F({\Phi
_k}v(j))),
\end{equation}
where $\hat J_i=\Phi _k^T{J_i}{\Phi _k}$ for $J_i=F'(u_i)$. The above equation could be rewritten in the form similar to
(\ref{generalerror}),
\begin{equation}\label{error}
v(j+1)=v(j)-(\Phi_k^T F'(\Phi_k^T
v(j))+\Delta(v(j)))^{-1}(\Phi_k^T F(\Phi_k v(j))+\epsilon(v(j)) ),
\end{equation}
where $ \epsilon (v(j))= \Phi _k^T\hat F({\Phi _k}v(j))-\Phi_k^T F(\Phi_k v(j))$, and $\Delta (v(j))=\hat J_i-\Phi_k^T F'(\Phi_kv(j))\Phi_k$.

If $u_i, \Phi_k v(j) \in \mathcal{B}_{\Phi_k v_*}(\delta)$, one obtains $\|\epsilon (v(j))\|\le \| \hat F(\Phi_k v(j)) - F(\Phi_k v(j)) \| <\epsilon_F$, and $\|\Delta(v(j))\| \le \|F'(u_i)-F'(\Phi_kv(j))\|\le \gamma (\|e(0)\|+\|e(j)\|)$. Using lemma \ref{lem:newton}, and define $K_C:=\bar K(1+2 \gamma)$, one obtains
 \begin{equation}\label{chordcon}
\left\| {e(j + 1)} \right\| < \bar K\left\| {e(j)} \right\|\left(
{\gamma \left\| {{e(0)}} \right\| + (1 + \gamma )\left\| {e(j)}
\right\|} \right) + \bar K{\varepsilon _F} \le {K_C}\delta
\left\| {e(j)} \right\| + \bar K{\varepsilon _F}.
\end{equation}
Let $\delta$ small enough such that $K_C \delta <1$. It follows that $\|e(j)\|$ is bounded by $\bar K{\varepsilon _F}/(1-K_C\delta)$ as $j \to \infty$.
 \hfill
\end{proof}
\medskip

The reduced chord iteration inherits one advantage of the standard chord iteration, i.e., it does not  compute the Jacobian at each iteration. Therefore, the per-iteration cost of this method is lower than the cost from a reduced model formed by Newton iteration. Specifically, if $\alpha(m)$ denotes the cost of evaluating $m$ components of $F$, then the cost of approximating a nonlinear vector field in the online stage is $O(\alpha(m)+4mk)$ via the DEIM \cite{SorensenDC:10a}. Let $\gamma$ be the average number of nonzero entries per row of the Jacobian. If approximating a Jacobian is computed during the online stage, an extra $O(\alpha(m)+2mk + 2\gamma mk + 2mk^2)$ cost is need via the DEIM in the simple case when $F$ is evaluated componentwise at $u$ \cite{SorensenDC:10a}. In the worst case, the complexity of computing $\hat J$ would  still depend on $n$ if the $J$ is dense. Even in
the best scenario when $J$ is diagonal,  the computational cost of a Jacobian matrix is  no less than the cost of the corresponding nonlinear vector field.

On the other hand, a reduced Newton iteration cannot have quadratic convergence rate since the DEIM approximation error of the Jacobian matrix or the vector field cannot be ignored. Therefore, the reduced chord iteration is more efficient for solving a large-scale nonlinear algebraic equation in general.

Notice that if the error bound of DEIM approximation, $\varepsilon_F$ in (\ref{chordcon}), approaches zero, the reduced chord iteration converges linearly to $v_*$.  Moreover, $\varepsilon_F$ is bounded by a constant times $\|(I-\Psi_m\Psi_m^T)F\|$~\cite{SorensenDC:10a}, which indicates an optimal collateral subspace is desired to decrease $\varepsilon_F$.

\subsection{Numerical Example}
In this subsection, the ARM is applied to an elliptic PDE (from \cite{PateraAT:06a} and \cite{SorensenDC:10a}),
\begin{equation}\label{staticeq}
 - {\nabla ^2}u(x, y) +
 \frac{\mu_1}{\mu_2}({{\rm{e}}^{\mu_2 u}} - 1)
 = 100\cos (2\pi x)\cos (2\pi y),
\end{equation}
with homogeneous Dirichlet boundary conditions, $u(0,y) = u(1,y) = u(x,0) = u(x,1) = 0$. The spatial variables $(x, y) \in \Omega= [0,1]^2$ and the parameters satisfy $\mu = (\mu_1, \mu_2) \in \mathcal{D} = [0.01,10]^2 \subset \mathbb{R}^ 2$. The ``real'' solution is solved by the Newton iteration resulting from a finite difference discretization. The spatial grid points $(x_i , y_j)$ are equally spaced in $\Omega$ for $i, j = 1,\ldots,50$. The full dimension for the state variable $u$ is then $n = 2500$. In the offline stage,
the reduced models are constructed based on 121 precomputed snapshots corresponding to 121 input parameters that are uniformly distributed in the parameter domain.   Thus, the parameter domain $\mathcal{D}$ is uniformly partitioned into 121 subdomains.
In the online stage,  all the reduced models are tested based on 200 randomly selected parameters.

\begin{figure}
\begin{center}
\begin{minipage}{0.48\linewidth} \begin{center}
\includegraphics[width=1\linewidth]{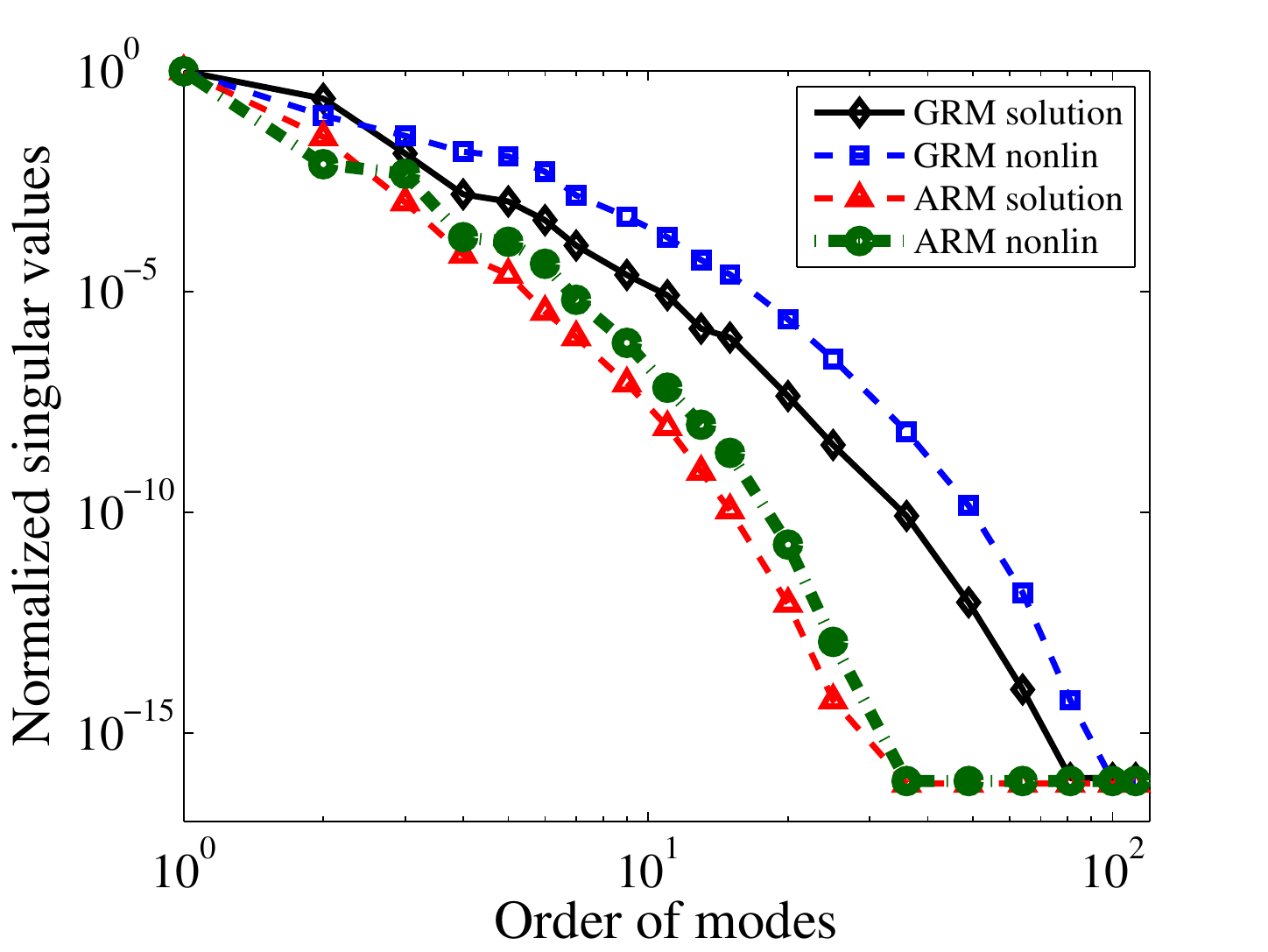}
\end{center} \end{minipage}\\
 \caption{(Color online.) Simulation results for the elliptic PDE (\ref{staticeq}). (a) Normalized singular values of the matrices of solution snapshots and nonlinear snapshots based on 121 precomputed parameters. The average value of the singular values in  the adaptive reduced model (ARM) decrease faster than the singular values in the global reduced model (GRM).} \label{fig:static}\vspace{-3mm}
\end{center}
\end{figure}

Figure \ref{fig:static} shows the normalized singular values for the solution snapshots of (\ref{staticeq}) and nonlinear snapshots $s(u,\mu)=\mu_1/\mu_2(\exp(\mu_2 u)-1)$ for  the precomputed data ensemble.  Since each subdomain has a different snapshot matrix for the ARM case, we compute the average value for the singular values.  All the singular values are normalized by  the largest one. Compared with the single values from GRM, those in ARM have a faster decreasing rate, which means that in order to obtain the same level of accuracy, ARM needs fewer modes to represent the original system. We also observe that singular values of solution matrices decrease faster than the singular values of nonlinear matrices. Thus, we set the dimension of nonlinear-term subspace twice as the dimension of the solution subspace for each individual test such that the DEIM approach can provide a good approximation for the original POD.

\begin{figure}
\begin{center}
\begin{minipage}{0.48\linewidth} \begin{center}
\includegraphics[width=1\linewidth]{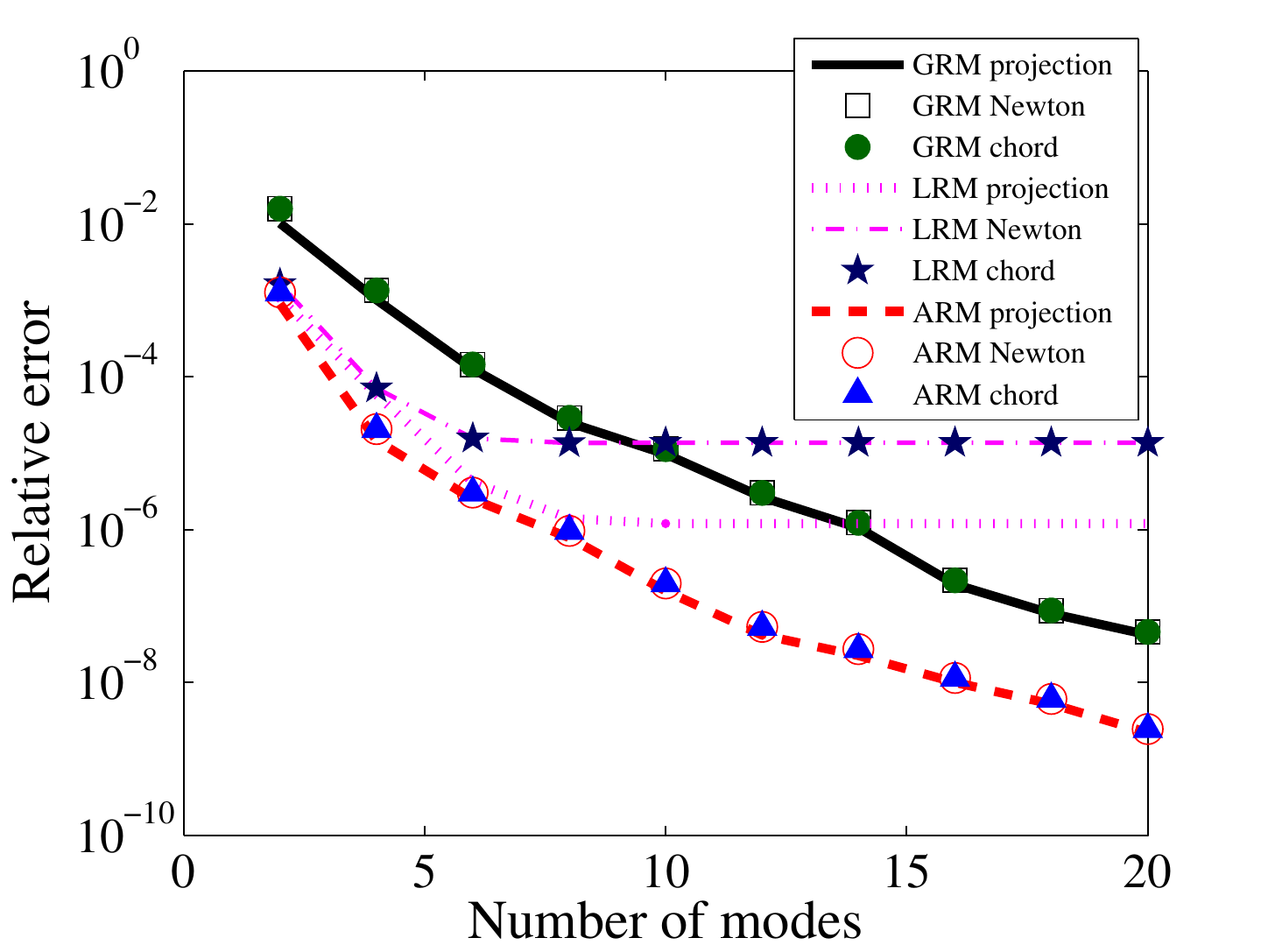}
\end{center} \end{minipage}
\begin{minipage}{0.48\linewidth} \begin{center}
\includegraphics[width=1\linewidth]{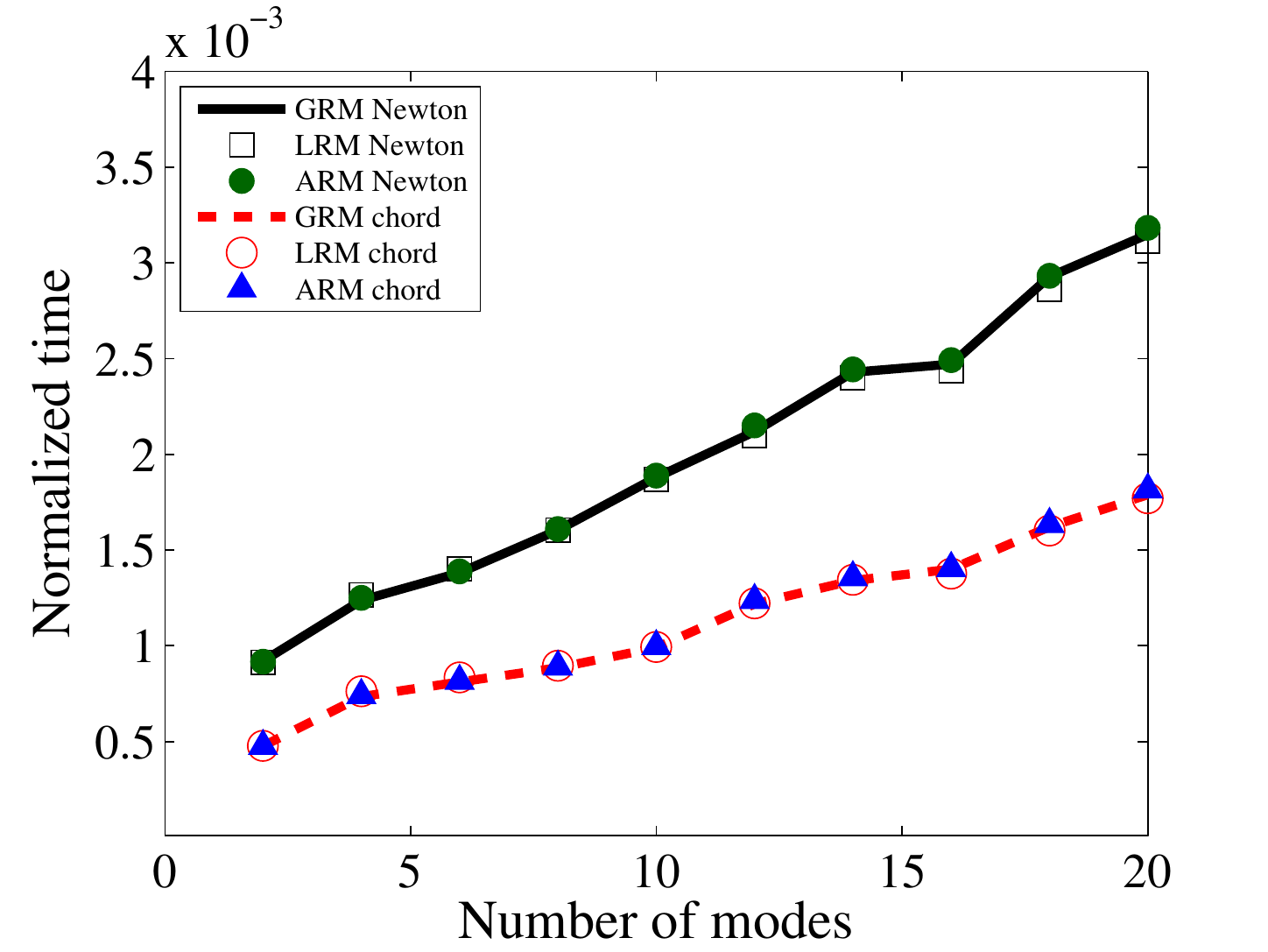}
\end{center} \end{minipage}\\
\begin{minipage}{0.48\linewidth}\begin{center} (a) \end{center}\end{minipage}
\begin{minipage}{0.48\linewidth}\begin{center} (b) \end{center}\end{minipage}
 \caption{(Color online.) (a) The relative error, $\left\|u(\mu ) - \hat u(\mu )\right\|/ \left\|u(\mu)\right\|$, for the reduced Newton iteration and the reduced chord iteration for the elliptic PDE (\ref{staticeq}). These errors are averaged over a set of  200 randomly selected parameters $\mu$ that were not used to obtain the sample snapshots. When the DEIM approximation is used, we set the subspace dimension for the nonlinear term $s(u,\mu)$ twice as the subspace dimension for the solution $u$ to balance the POD error and the error from DEIM approximation. Therefore, the number of POD modes $k$ is  $2, 4, \ldots, 20$, and the number of the nonlinear-term modes is $4, 8, \ldots 40$. The projection error, $e_k$, denotes relative error between $u$ and its projection onto the subspace $\mathcal{S}_k$. (b) Average running time (scaled with the running time for the full system) for each reduced Newton iteration and reduced chord iteration.} \label{fig:modes}\vspace{-3mm}
\end{center}
\end{figure}

Figure \ref{fig:modes}(a) shows that reduced chord iteration could obtain the same accuracy as the reduced Newton iteration, for all GRM, LRM, and ARM cases. On the other hand, the DEIM reduced system formed by the ARM has a much higher accuracy than GRM and LRM for the same dimension.
When $k$ is relatively small, the SVD-based truncation error $e_o$ is the dominant term of $e_k$; as  $k$ increases,  $e_o$ diminishes and finally the projection error $e_r$ takes the dominance in $e_k$. For the LRM case, we choose 9 solution snapshots and 9 nonlinear term snapshots for each local basis.
Hence, compared with the GRM,  the LRM yields a more accurate solution when $k$ is small and a less accurate solution when $k$ is large.
 Especially, when $k\ge 9$, $e_o=0$ holds exactly for the LRM. However, since the LRM has a larger $e_r$ than the other two methods,  it cannot obtain a high accuracy solution even though $k$ is very large. Moreover, for the LRM case, the numerical error of the reduced Newton method and the reduced chord method can not approach to $e_k$, which implies that the DEIM approximation based on 9 effective modes yielders an addition error $e_i$. On the contrary, the ARM can balance $e_o$ and $e_r$ for a wide range of subspace dimensions, which yield a better solution.

 Figure \ref{fig:modes}(b) shows the scaled running time
for the reduced Newton method and the reduced chord method. Both approaches obtain more than 200 times speedups for each iteration.
 For each iteration, the reduced chord iteration only update $n$ elements in the  vector field, while the reduced Newton iteration update an extra $n$ nonzero elements in the Jacobian matrix.  Thus, one can expect that the reduced chord iteration takes about half time compared with reduced Newton iteration.  This  expectation is also verified by the simulation.

\begin{figure}
\begin{center}
\begin{minipage}{0.32\linewidth} \begin{center}
\includegraphics[width=1\linewidth]{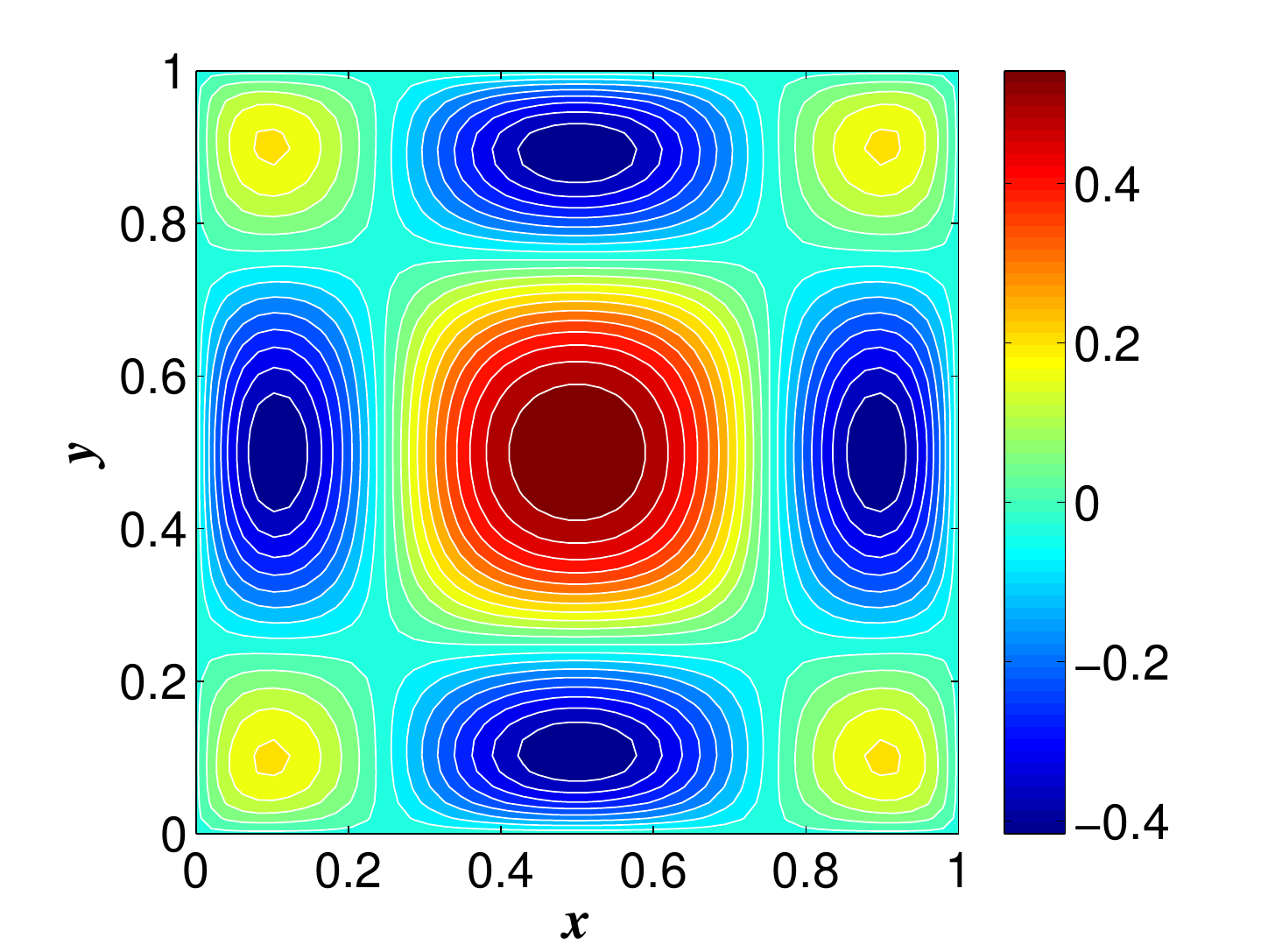}
\end{center} \end{minipage}
\begin{minipage}{0.32\linewidth} \begin{center}
\includegraphics[width=1\linewidth]{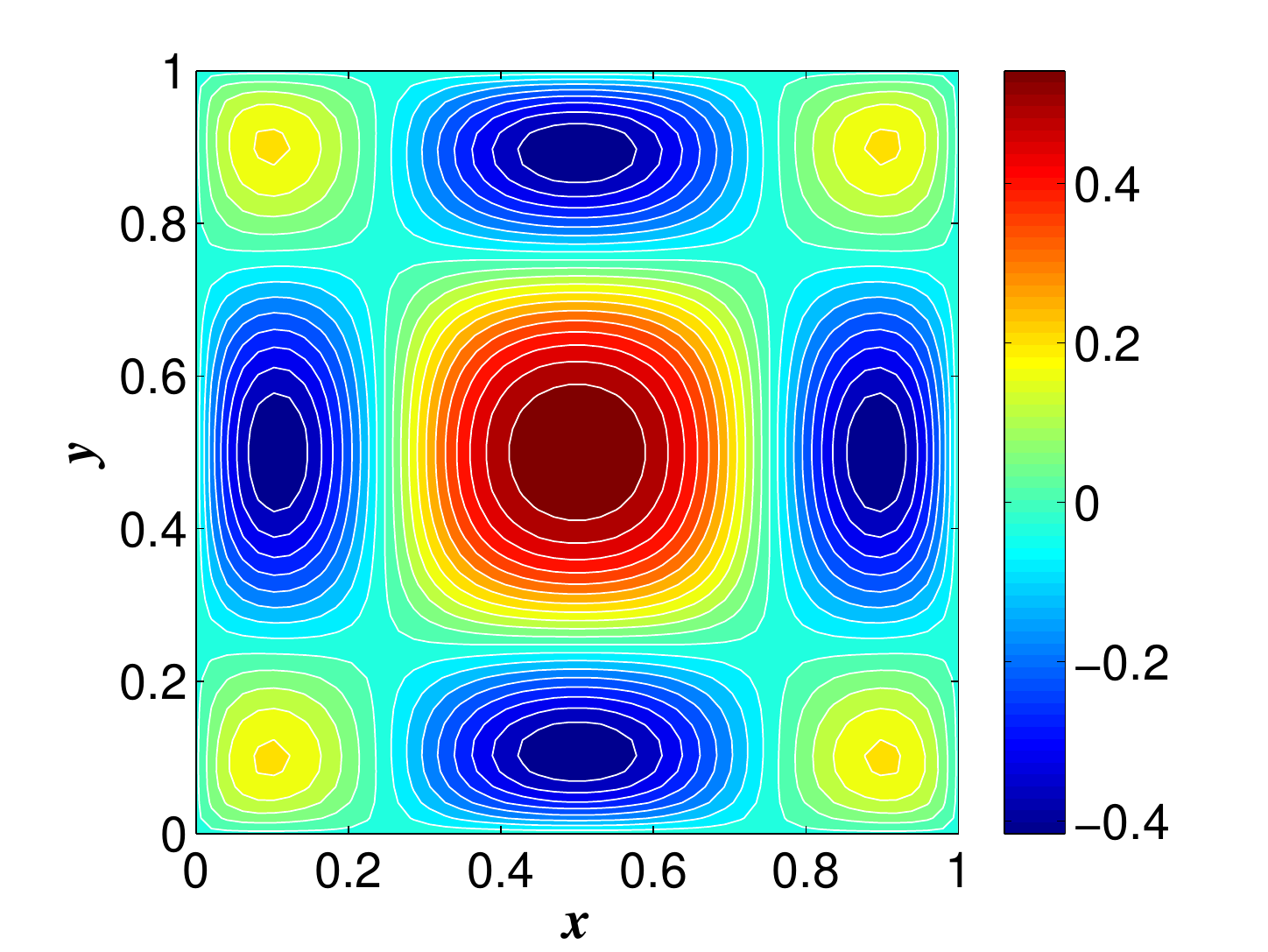}
\end{center} \end{minipage}
\begin{minipage}{0.32\linewidth} \begin{center}
\includegraphics[width=1\linewidth]{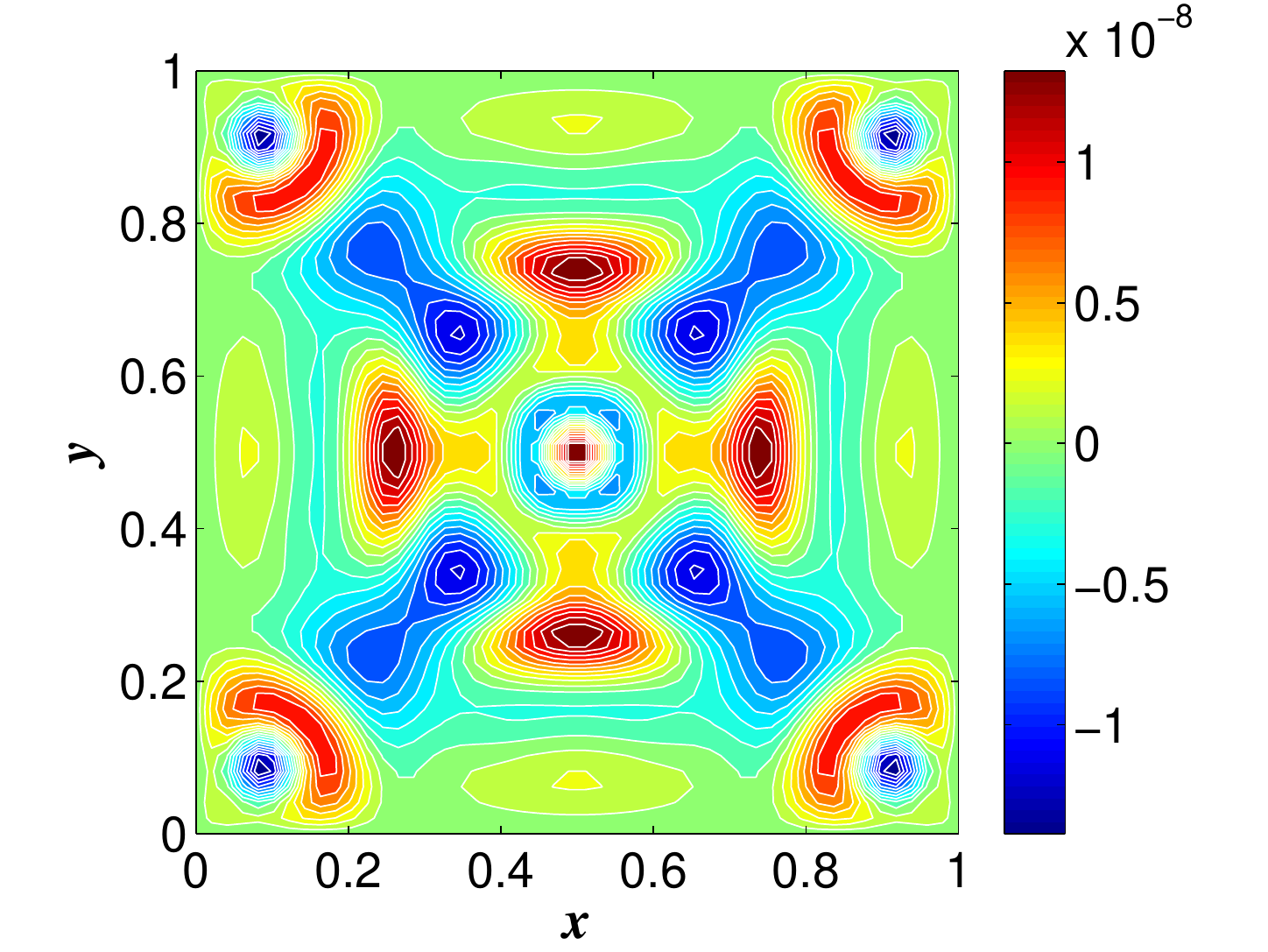}
\end{center} \end{minipage}\\
\begin{minipage}{0.32\linewidth}\begin{center} (a) \end{center}\end{minipage}
\begin{minipage}{0.32\linewidth}\begin{center} (b)\end{center}\end{minipage}
\begin{minipage}{0.32\linewidth}\begin{center} (c)\end{center}\end{minipage}
\caption{(Color online.) Simulation results for the elliptic PDE (\ref{staticeq}) with $\mu = (\mu_1, \mu_2) = (4.5,8.5)$. (a) The benchmark solution solved by the full model with 2500 grid points. (b)  The approximating solution solved by the ARM-chord method with $k=10$. (c) Numerical error $e=u-\hat u$ of the ARM-chord reduced system with $k=10$. } \label{fig:staticsnap}\vspace{-3mm}
\end{center}
\end{figure}

Figure \ref{fig:staticsnap}(a) shows the solution corresponding to the input parameters $\mu_1=4.5$, and $\mu_2=8.5$. The ARM-chord reduced system has a very good approximation for the original system with $k=10$, and the solution profile is given in Figures \ref{fig:staticsnap}(b). The numerical error $u-\hat u$ is given in  Figure \ref{fig:staticsnap}(c).

Finally, we study the numerical error of the  ARM chord method with different subspace $k$ and the kernel length $\sigma$.
Table \ref{tab:ksigma} indicates that the ARM error is not very sensible with $\sigma$; with a large range of $\sigma$ ($\sigma=1  \sim 8$), the relative error for each $k$ is no greater than 2 times of the minimal error with the optimal  $\sigma$. Moreover, as $k$ increases, the optimal  $\sigma$ tends to shift to a larger value.

\begin{table} [htbp]
\begin{center}
\caption{The numerical error of the ARM chord method with different subspace $k$ and the kernel length $\sigma$.}
 \label{tab:ksigma}
\begin{tabular}{r|cccccccc}
& \multicolumn{8}{c}{$\sigma$} \\
\cline{2-9}
 $k$ &  0.25&	0.5&	1&	2&	4&	6&	8&	10  \\
%\midrule [0.2pt]
\hline
2&  1.27E-03&	1.27E-03&	1.27E-03&	1.36E-03&	1.64E-03&	2.04E-03&	2.41E-03&	2.64E-03\\
4&  2.54E-05&	2.23E-05&	2.08E-05&	3.71E-05&	7.70E-05&	1.28E-04&	1.59E-04&	1.79E-04\\
6&  7.99E-06&	3.20E-06&	3.05E-06&	5.48E-06&	1.01E-05&	1.35E-05&	1.83E-05&	2.20E-05\\
8&  1.76E-06&	1.57E-06&	9.64E-07&	7.03E-07&	1.27E-06&	1.87E-06&	2.63E-06&	3.32E-06\\
10& 5.74E-07&	3.15E-07&	1.99E-07&	1.68E-07&	2.43E-07&	4.24E-07&	6.49E-07&	8.96E-07\\
12& 3.24E-07&	6.77E-08&	5.33E-08&	5.70E-08&	9.27E-08&	1.35E-07&	1.90E-07&	2.49E-07\\
14& 3.27E-07&	4.80E-08&	2.72E-08&	2.09E-08&	4.16E-08&	7.00E-08&	9.46E-08&	1.15E-07\\
16& 4.05E-07&	3.11E-08&	1.15E-08&	8.40E-09&	8.76E-09&	1.31E-08&	1.98E-08&	2.71E-08\\
18& 3.15E-07&	9.51E-08&	6.12E-09&	4.29E-09&	4.24E-09&	4.76E-09&	6.21E-09&	8.50E-08\\
20& 3.85E-07&	2.35E-08&	2.46E-09&	1.71E-09&	1.46E-09&	2.22E-09&	2.91E-09&	3.63E-09\\
\end{tabular}
\end{center}
\end{table}

\section{Nonlinear Parabolic PDEs}\label{sec:parabolic}
In this section, we extend our adaptive model reduction approach to treat nonlinear parabolic PDEs. We first introduce the methodology and then illustrate that the proposed method can yield a high accuracy for the nonlinear Navier-Stokes equation in a low dimensional space.

\subsection{Methodology}
The general form of parabolic PDEs, after discretization, is given by an ODE (\ref{ode}). We still follow the offline-online splitting computational strategy: In the offline computation, for each input parameter $\mu_i$ the solution trajectory gives a snapshot matrix $X_i=[u(t_1,\mu_i),\ldots, u(t_T,\mu_i)]$. Using the standard SVD, we have $X_i=V_i \Lambda_i W_i^T $, where $V_i = [{v_1}({\mu_i}),\ldots,{v_{r_i}}({\mu _i})]$ contains a set of orthonormal basis vectors, $\Lambda_i = {\rm{diag}}\{ {\lambda _1}({\mu _i}),\ldots,{\lambda_{r_i}}({\mu _i})\}$ contains the corresponding eigenvalues. The POD basis matrix $\Phi_i$ is given by the first $k_i$ columns of $V_i$. Correspondingly $\Lambda'_i$ is defined as the first $k_i\times k_i$ block of $\Lambda_i$. Then $\Phi_i$ minimizes the truncation error of $ X_i$ and its projection onto the column space of $\Phi_i$, which is given by $E_i={\left\|{(I - {\Phi_i }{\Phi_i}^T) X_i} \right\|_F}$.

We shall use adaptive reduced bases to form a reduced model. As an analogy of (\ref{info}), a weighted snapshot matrix for  subdomain $i$ can be defined by
\begin{equation}\label{informatrix}
X^A_i=[a_{i1} X_1, \ldots, a_{iN} X_N],
\end{equation}
where $\{a_{ij}\}_{j=1}^N$ are the weighting coefficients for the subdomain $i$.  Similar to the method for elliptic PDEs, we implicitly partition the interested parameter region into some subdomains and pre-compute the adaptive POD basis for each subdomain in the offline stage. Especially, if all $a_{ij}$s equal 1, the ARM degenerates to the GRM and if $a_{ij}$ is formed by a compact scheme, such as a Gaussian function, it leads to less truncation error during the SVD process. For convenience, we shall remove the subscript $i$ for convenience and use $X^A$ in place of $X^A_i$.

The direct SVD of $X^A$ could be employed to obtain a reduced basis, but it is not the most efficient approach. When the trajectories exhibit fast variation over a long time domain, a great deal of memory must be allocated to record $X^A$. Suppose we pick $T$ snapshots from each trajectory, then the total size of $X^A$ is $n\times NT$. If $NT$ is very large, the SVD of $X^A$ could be expensive. In order to save more memory and improve the efficiency, the online POD basis could be constructed from the first fewer POD modes of some precomputed trajectories, rather than the original snapshots \cite{MadayY:12a,GreplMA:05a}. In order to adaptively pick these trajectories, we define an \emph{information matrix},
\begin{equation}\label{infoburgers}
 X' := [{a_1}\Phi_1 \Lambda _1',\ldots,a_N\Phi_N\Lambda_N'].
 \end{equation}
If each $\Phi_k$ is a $n\times k$ matrix, then the size of $X'$ is $n\times Nk$. Notice that $k\ll T$ for parabolic PDEs with large time domain. We claim the adaptive reduced bases can be computed by the SVD of $X'$ with a small truncation error.

\medskip
\begin{lemma} \label{lem:truncel}
Let $\Phi\in \mathcal{V}_{n,k'}$ minimize the truncation error of $X'$ and its projection onto a $k'$-dimensional subspace of $\mathbb{R}^n$. The error is given by $E_0={\left\| {(I - {\Phi}{\Phi }^T) X'} \right\|_F}$ in Frobenius norm. Then the projection error for the original weighted snapshot matrix $X^A$ in (\ref{informatrix}) is bounded by
\begin{equation}\label{conc}
{\left\| {(I - {\Phi }{\Phi }^T) X^A} \right\|_F}
 \le {E_0} + \sqrt {\sum\limits_{i = 1}^N {a_i^2E_i^2} }.
\end{equation}
\end{lemma}
\begin{proof}
Since the non-truncated SVD gives $X_i=V_i \Lambda_i W_i^T$, by the definition of $X^A$ in (\ref{informatrix}), we have $X^A=[a_1 V_1 \Lambda_i W_1^T, \ldots,a_N V_N \Lambda_N W_N^T]$. We construct a new matrix
\begin{equation*}
\tilde X^A:=[a_1 V_1 \tilde \Lambda _1 W_1^T, \ldots, a_N V_N \tilde
\Lambda _N W_N^T],
\end{equation*}
where $\tilde \Lambda_i$ is has the same size as $\Lambda_i$, but only contains the first $k_i$ nonzero singular values of $\Lambda_i$, i.e, $\tilde \Lambda_i ={\rm{diag}}\{ {\lambda _1}({\mu_i}),\ldots,{\lambda _{k_i}}({\mu _i}), 0,\ldots,0\}$. Since $W_i$ are orthonormal,
\begin{equation}
(X^A-\tilde X^A) (X^A-\tilde X^A)^T=\sum_{i=1}^N a_i^2 V_i
(\Lambda_i-\tilde \Lambda_i)^2 V_i^T.
\end{equation}
It follows that
\begin{equation} \label{con1}
\begin{array}{l}
 \left\| {X^A - \tilde X^A} \right\|_F^2 = \mathrm{tr}\left( {(X^A - \tilde X^A){{(X^A - \tilde X^A)}^T}} \right)
 = \sum\limits_{i = 1}^N {\mathrm{tr}(a_i^2{V_i}{{({\Lambda _i} - {{\tilde \Lambda }_i})}^2}V_i^T)} \\
 \ \ \ \ \ \ \ \ \ \ \ \ \ \ \ \ = \sum\limits_{i = 1}^N {a_i^2\mathrm{tr}({{({\Lambda _i} - {{\tilde \Lambda }_i})}^2})} = \sum\limits_{i = 1}^N {a_i^2E_i^2} \\
 \end{array}
 \end{equation}
The last equity holds because $E_i={\left\| {(I - {\Phi_i }{\Phi_i}^T) X_i} \right\|_F}=\sqrt{\sum_{j=k_i+1}^{r_i}\lambda_j^2}$. On the other hand,
\begin{equation} \label{con2}
\begin{array}{l}
 {\left\| {(I - \Phi {\Phi ^T})\tilde X^A} \right\|_F} = {\left\| {(I - \Phi {\Phi ^T})[{a_1}{V_1}{{\tilde \Lambda }_1},\ldots,{a_N}{V_N}{{\tilde \Lambda }_N}]{\rm{diag}}\{ W_1^T,\ldots,W_N^T\} } \right\|_F} \\
 = {\left\| {(I - \Phi {\Phi ^T})[{a_1}{V_1}{{\tilde \Lambda }_1},\ldots,{a_N}{V_N}{{\tilde \Lambda }_N}]} \right\|_F} = {\left\| {(I - \Phi {\Phi ^T})[{a_1}{\Phi _1}{\Lambda _1'},\ldots,{a_N}{\Phi _N}{\Lambda _N'}]} \right\|_F} \\
 = {\left\| {(I - \Phi {\Phi ^T}){{ X'}}} \right\|_F} = {E_0} \\
 \end{array}.
 \end{equation}
Using (\ref{con1}) and (\ref{con2}), and combining the following inequity,
\begin{equation}
\begin{array}{l}
 {\left\| {(I - \Phi {\Phi ^T})X^A} \right\|_F} \le {\left\| {(I - \Phi {\Phi ^T})\tilde X^A} \right\|_F} + {\left\| {(I - \Phi {\Phi ^T})(X^A - \tilde X^A)} \right\|_F} \\
 \ \ \ \ \ \ \ \ \ \ \ \ \ \ \ \ \ \ \ \ \ \ \ \le {\left\| {(I - \Phi {\Phi ^T})\tilde X^A} \right\|_F} + {\left\| {X^A - \tilde X^A} \right\|_F} \\
 \end{array}
 \end{equation}
 ($\ref{conc}$) is obtained.
 \hfill
\end{proof}

\medskip

We notice that each $\Phi_i$ in (\ref{infoburgers}) does not necessary to have the same size. Based on the singular values in $\Lambda_i$, we can adaptively set $k_i$ such that each $E_i$ is smaller than a constat number, say $\epsilon_0$, during the offline stage. Moreover, the dimension $k'$ of $\Phi$ could also be adaptively chosen such that the total projection error bound, given by the RHS of (\ref{conc}), is smaller than another given number.

Similarly the collateral information matrix for the nonlinear snapshots shares the similar expression, i.e.,
\begin{equation}\label{infoburgers1}
Y' = [{a_1}\Psi_1\Sigma_1,\ldots,a_N\Psi_N\Sigma_N],
 \end{equation}
where $\Psi_i$ and $\Sigma_i$ are computed by the truncated SVD for the nonlinear snapshots on the trajectory $u(\mu_i)$. Having two sets of basis vectors, a reduced equation, formed by POD-Galerkin approach (\ref{rompara}) or POD-DEIM approximation (\ref{deimapprox}), could be used to solve an approximating solution $\hat u(t,\mu)$.

\subsection{Cavity Flow Problem}\label{sec:CFllp}
In this section, the performance of the adaptive model reduction is illustrated in numerical simulation of  the Navier-Stokes equation in a lid-driven cavity flow problem. We focus on demonstrating the capability  of the adaptive model reduction algorithms to deliver  accurate solutions with significant speedups.

Mathematically, the cavity-flow problem can be represented in terms of the stream function $\psi$ and vorticity $\omega$ formulation of the incompressible Navier-Stokes equation. In non-dimensional form, the governing equations are given as
\begin{equation} \label{sf}
\psi_{xx} + \psi_{yy} =  - \omega ,
\end{equation}\vspace{-5mm}
\begin{equation} \label{vt}
\omega_t =  - \psi_y  \omega_x + \psi_x\omega_y +
\frac{1}{\rm{Re}}\left( \omega_{xx} + \omega_{yy} \right),
\end{equation}
where Re is the Reynolds number and $x$ and $y$ are the Cartesian coordinates.  The space domain $\Omega=[0,L_x] \times [0,L_y] $ is fixed in time for each test. The velocity field is given by $u=\partial \psi /\partial y$, $v=-\partial \psi /\partial x$. No-slip boundary conditions are applied on all nonporous walls including the top wall moving at speed $U=1$. Using Thom's formula~\cite{ThomA:33a}, these conditions are, then, written in terms of stream function and vorticity. For example on the top wall one might have
\begin{equation}
\psi_B=0,
\end{equation}
\begin{equation}
\omega_B=\frac{-2\psi_{B-1}}{h^2}-\frac{U}{h},
\end{equation}
where  subscript $B$ denotes points on the moving wall, subscript $B-1$ denotes points adjacent to the moving wall, and $h$ denotes grid spacing. Expressions for $\psi$ and $\omega$ at remaining walls with $U=0$ can be obtained in an analogous manner.
 The initial condition is set as $u(x,y)=v(x,y)=0$. The discretization is performed on a uniform mesh with second-order central finite difference approximations for second-order derivatives in (\ref{sf}) and (\ref{vt}). The convective term in (\ref{vt}) is discretized via a first-order upwind difference scheme. For the time integration of (\ref{vt}), the implicit Crank-Nicolson scheme is applied for the diffusion term, and the explicit two-step Adams-Bashforth method is employed for the advection term. Since the governing equation contains only linear and quadratic terms,  we apply the Galerkin projection to construct reduced models.

In the numerical simulation, the full model uses $129 \times 129$ grid points and $\delta t= 2\times10^{-3}$ as a unit time step. The offline computation varies the Reynolds number from $600$ to $1600$ with equally spaced intervals 200. The horizontal length is given by a fixed value $L_x=1$ while the vertical length varies from 0.8 to 1.2 with equally spaced intervals 0.1. For convenience, each of the  precomputed input parameter (Re, $L_y$) is used as a reference point for the parameter domain $  [500, 1700]  \times[0.75, 1.25]$.
In the weighting function (\ref{weight}), the distance ${\left\| {{\mu _i} - {\mu _*}}\right\|}^2$ in the parameter domain is defined as $({{Re_i-Re_*}\over{2000}})^2+(L_{y_i}-L_{y_*})^2$ such that the variations of the Reynolds number and the aspect ratio are measured on a similar scale.
For the online testing, we randomly select 100 parameters in the same parameter domain. Meanwhile, we also partition the whole spatial domain $[0, 50]$ into 10 segments, and each adaptive reduced basis is constructed from all the data from one segment and partial data from its neighbours. For example, the first segment takes the solution snapshot from the time domain [0,6], and the second segment takes solution snapshot from the time domain [4, 11].

Considering that the  states of the Navier-Stokes equation show high time dependencies, compared with an elliptic equation, more modes are needed in order to present the entire solution trajectory with high accuracy. For each trajectory segment, the first 80 modes with singular values of solution snapshots are restored. Figures \ref{fig:svdp} compares the normalized singular values of the information matrices for the GRM and the average values for the AGM. As expected, the singular values in the ARM decrease more quickly than those from the GRM, which means in order to obtain the same level of accuracy, the ARM needs less modes to represent the original system.

\begin{figure}
\begin{center}
\begin{minipage}{0.48\linewidth} \begin{center}
\includegraphics[width=1\linewidth]{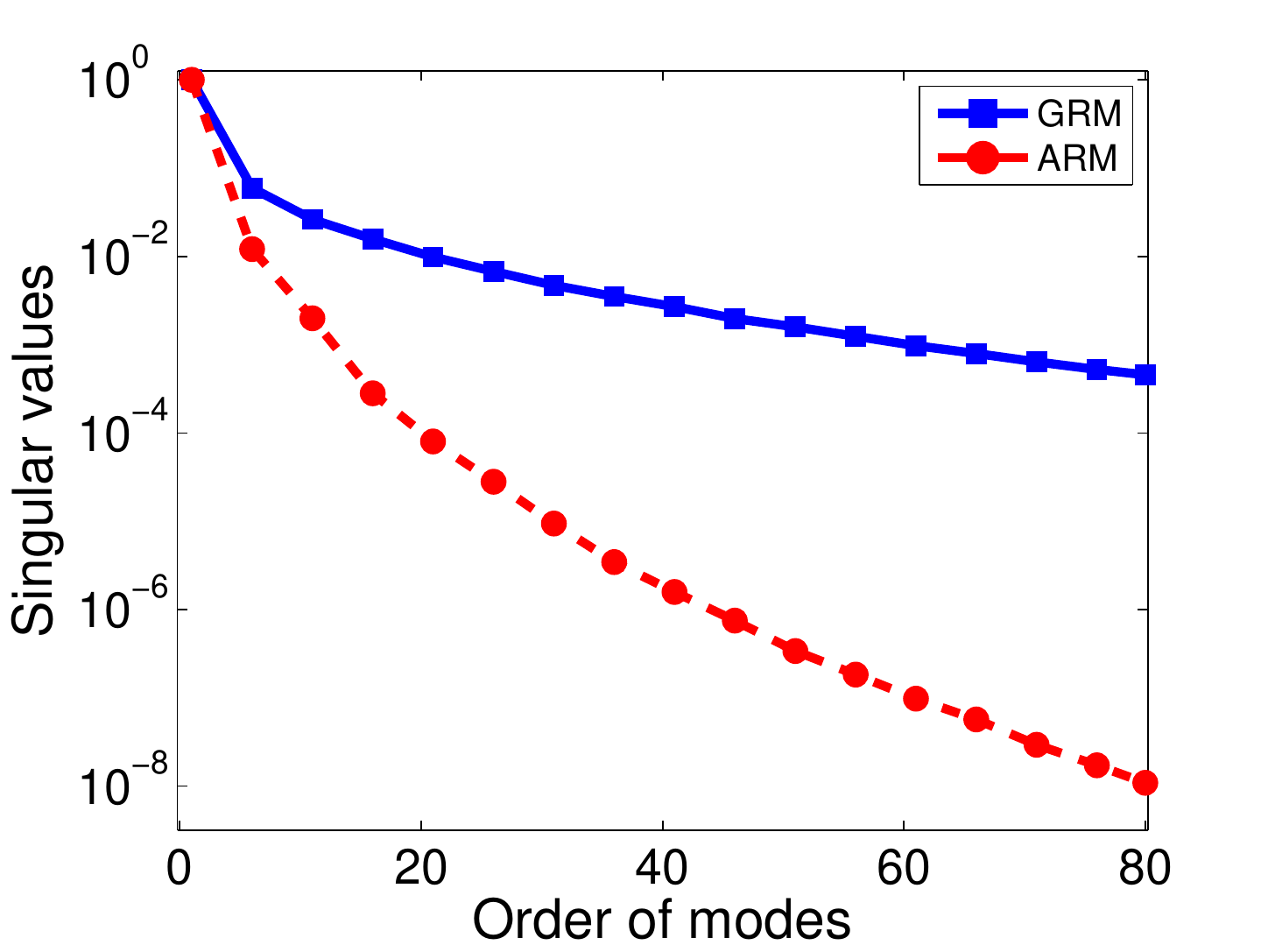}
\end{center} \end{minipage}\\
\caption{(Color online.)  Normalized singular values of the information matrices of the solution snapshots for the 2D Navier-Stokes equation in a lid-driven cavity flow problem. The singular values of the adaptive reduced basis have faster decreasing rate compared with the global reduced basis.} \label{fig:svdp}
\vspace{-3mm}
\end{center}
\end{figure}

\begin{figure}
\begin{center}
\begin{minipage}{0.32\linewidth} \begin{center}
\includegraphics[width=1\linewidth]{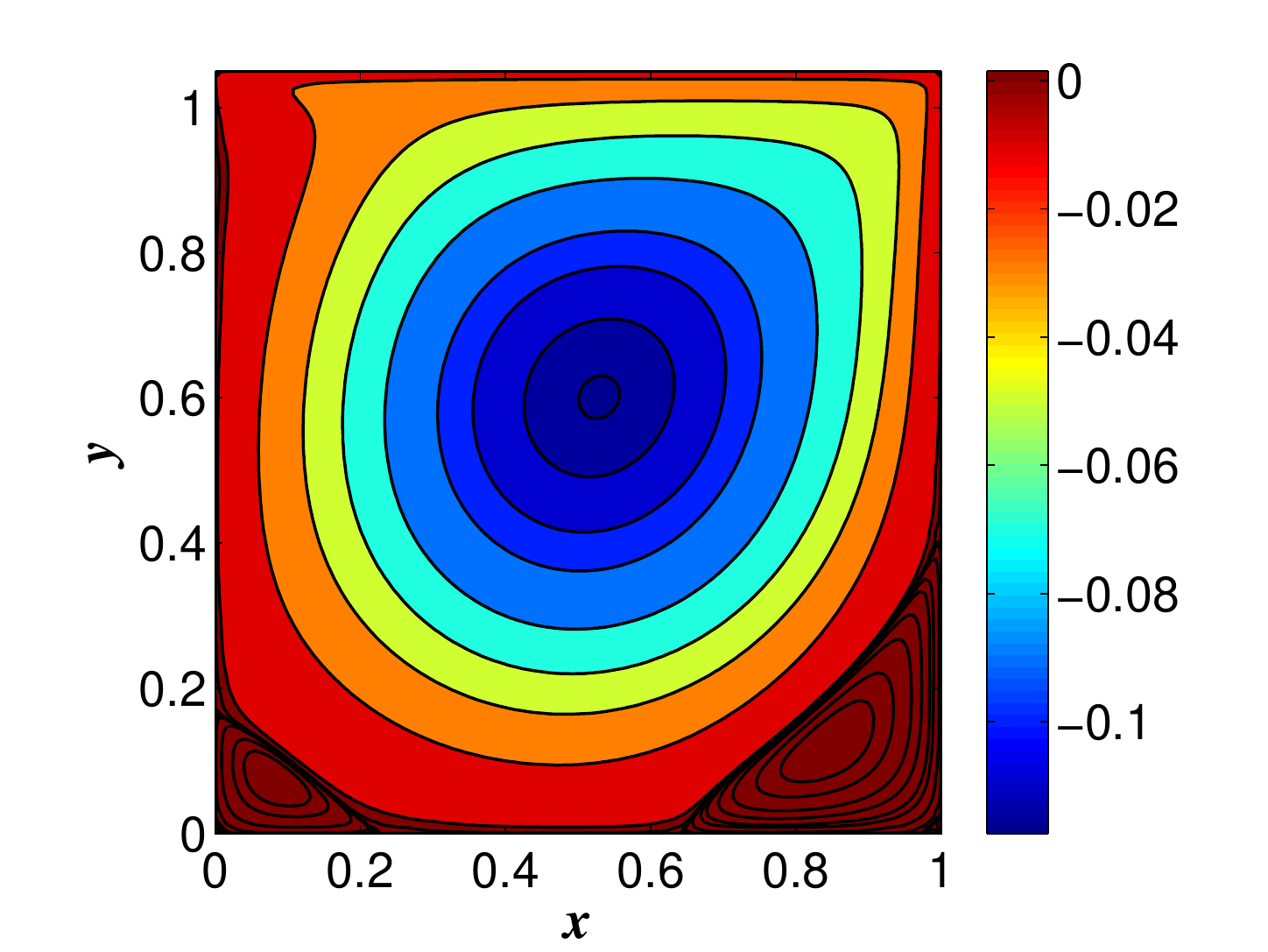}
\end{center} \end{minipage}
\begin{minipage}{0.32\linewidth} \begin{center}
\includegraphics[width=1\linewidth]{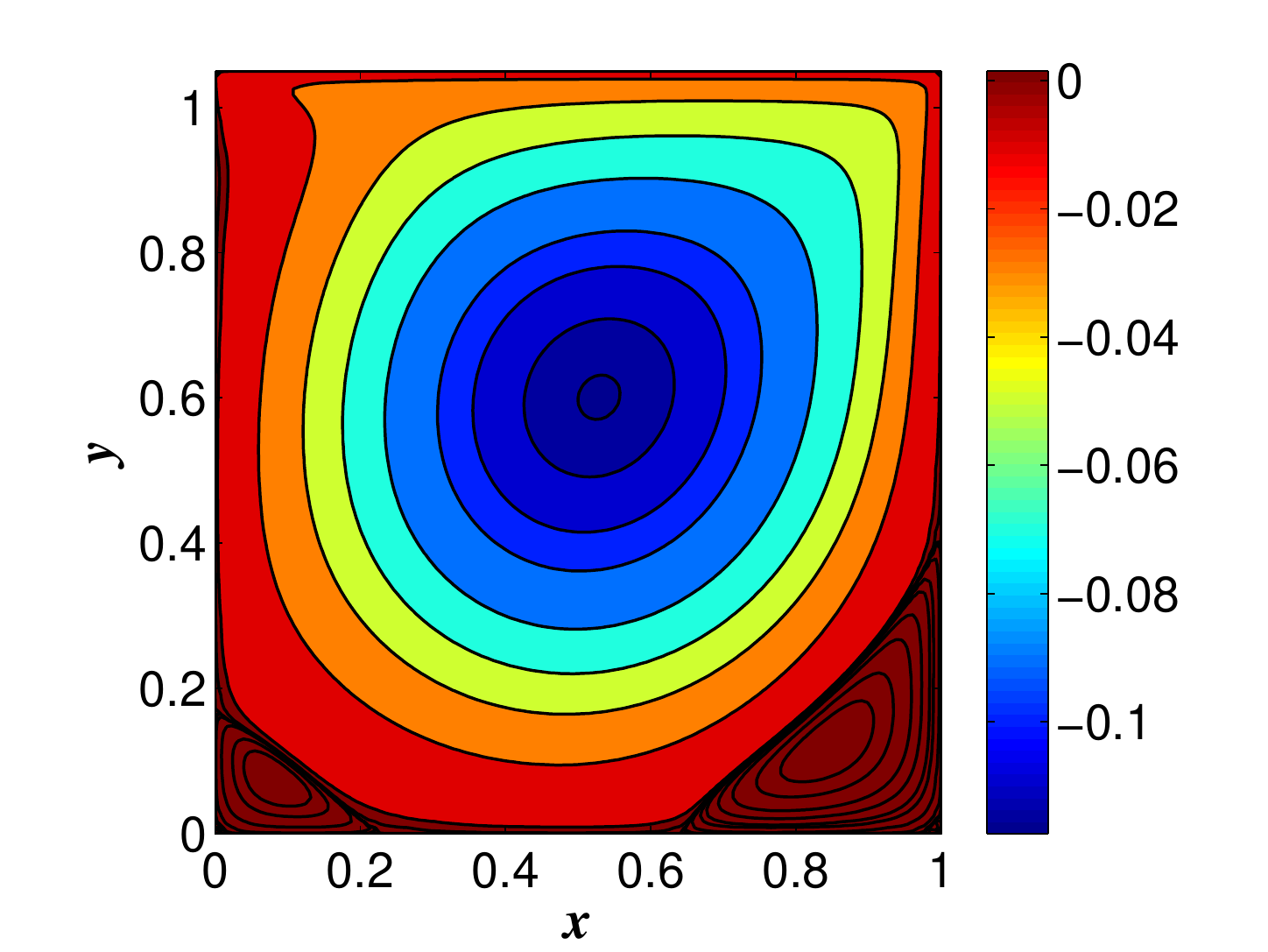}
\end{center} \end{minipage}
\begin{minipage}{0.32\linewidth} \begin{center}
\includegraphics[width=1\linewidth]{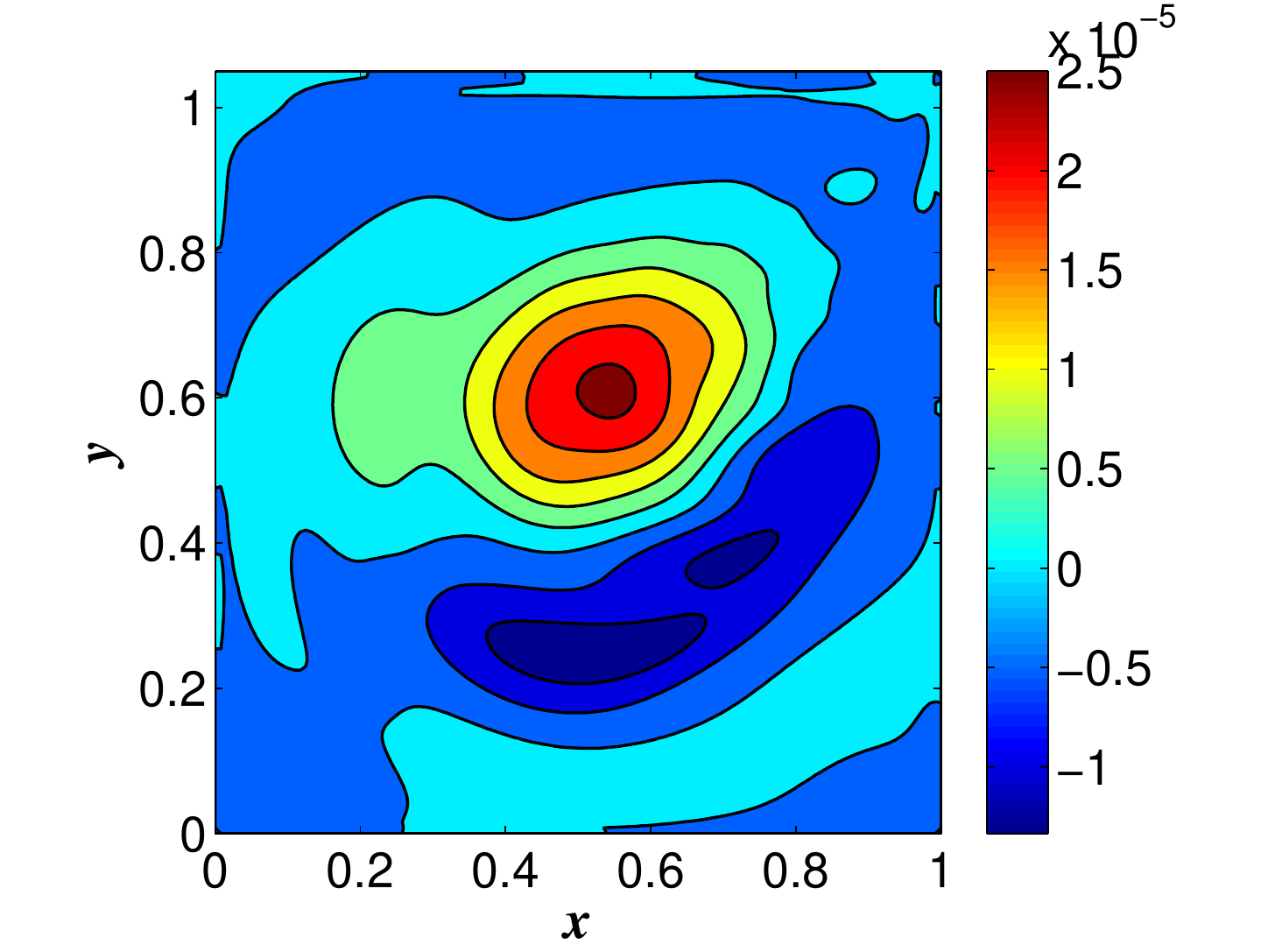}
\end{center} \end{minipage}
\begin{minipage}{0.32\linewidth}\begin{center} (a) \end{center}\end{minipage}
\begin{minipage}{0.32\linewidth}\begin{center} (b) \end{center}\end{minipage}
\begin{minipage}{0.32\linewidth}\begin{center} (c) \end{center}\end{minipage}
\caption{(Color online.) Streamline pattern for driven cavity problem with Re=$1050$ and $L_y=1.05$ at $t=50$. (a) The full model uses $129 \times 129$ grid points. (b) The approximating result obtained through ARM with 20 modes. We plot the contours of $\psi$ whose values are $-1\times 10^{-10}$, $-1\times 10^{-7}$, $-1\times 10^{-5}$, $-1\times 10^{-4}$, $-0.01$, $-0.03$, $-0.05$, $-0.07$, $-0.09$, $-0.1$, $-0.11$, $-0.115$, $-0.1175$, $1\times 10^{-8}$, $1\times 10^{-7}$, $1\times 10^{-6}$, $1\times 10^{-5}$, $5\times 10^{-5}$, $1\times 10^{-4}$, $2.5\times 10^{-4}$, $1\times 10^{-3}$, $1.3\times 10^{-3}$, and $3\times 10^{-3}$. (c) The error between $u-\hat u$ of the ARM.}
\label{fig:stream}
\end{center}
\end{figure}

\begin{figure}
\begin{center}
\begin{minipage}{0.45\linewidth} \begin{center}
\includegraphics[width=1\linewidth]{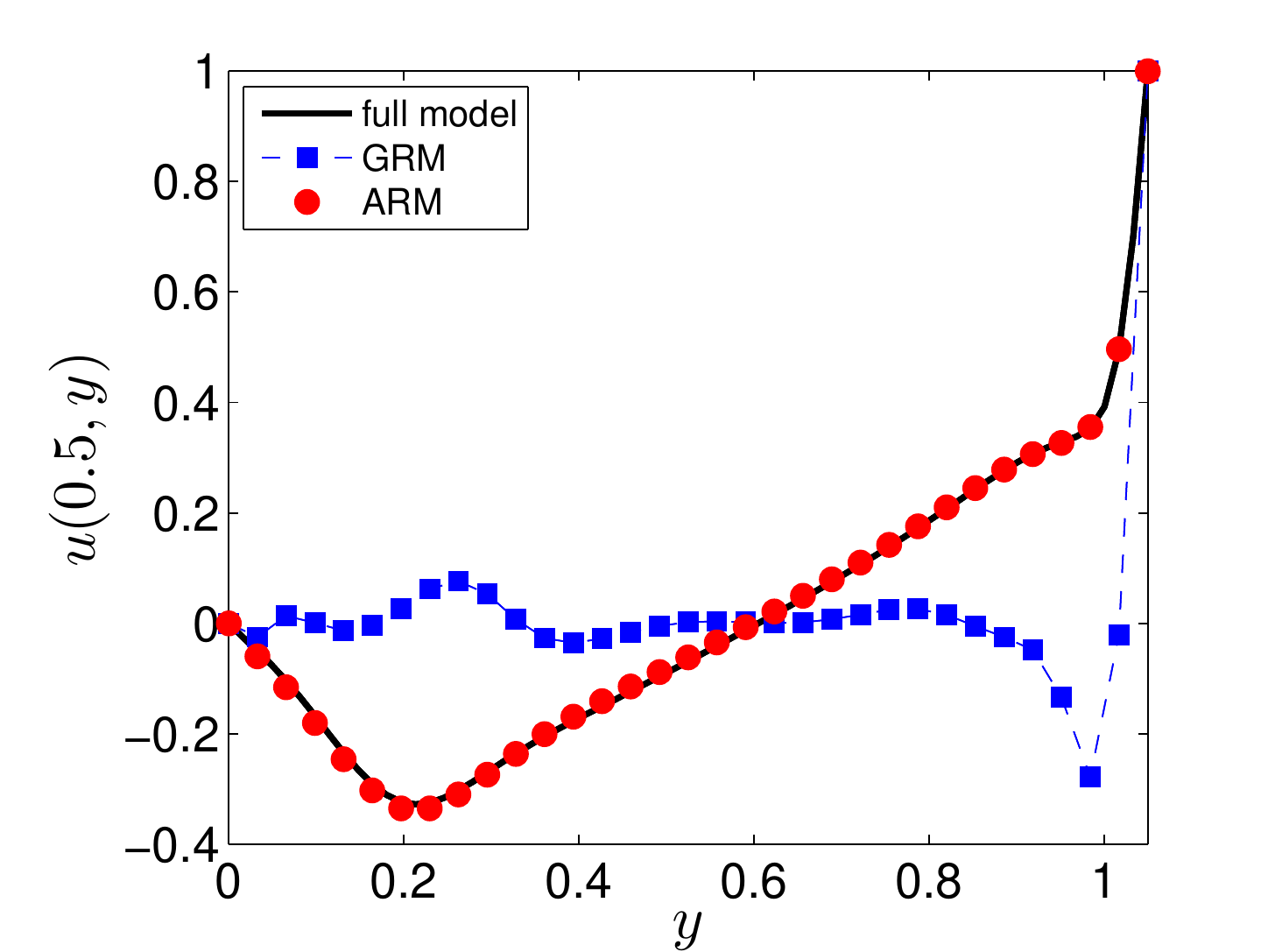}
\end{center} \end{minipage}
\begin{minipage}{0.45\linewidth} \begin{center}
\includegraphics[width=1\linewidth]{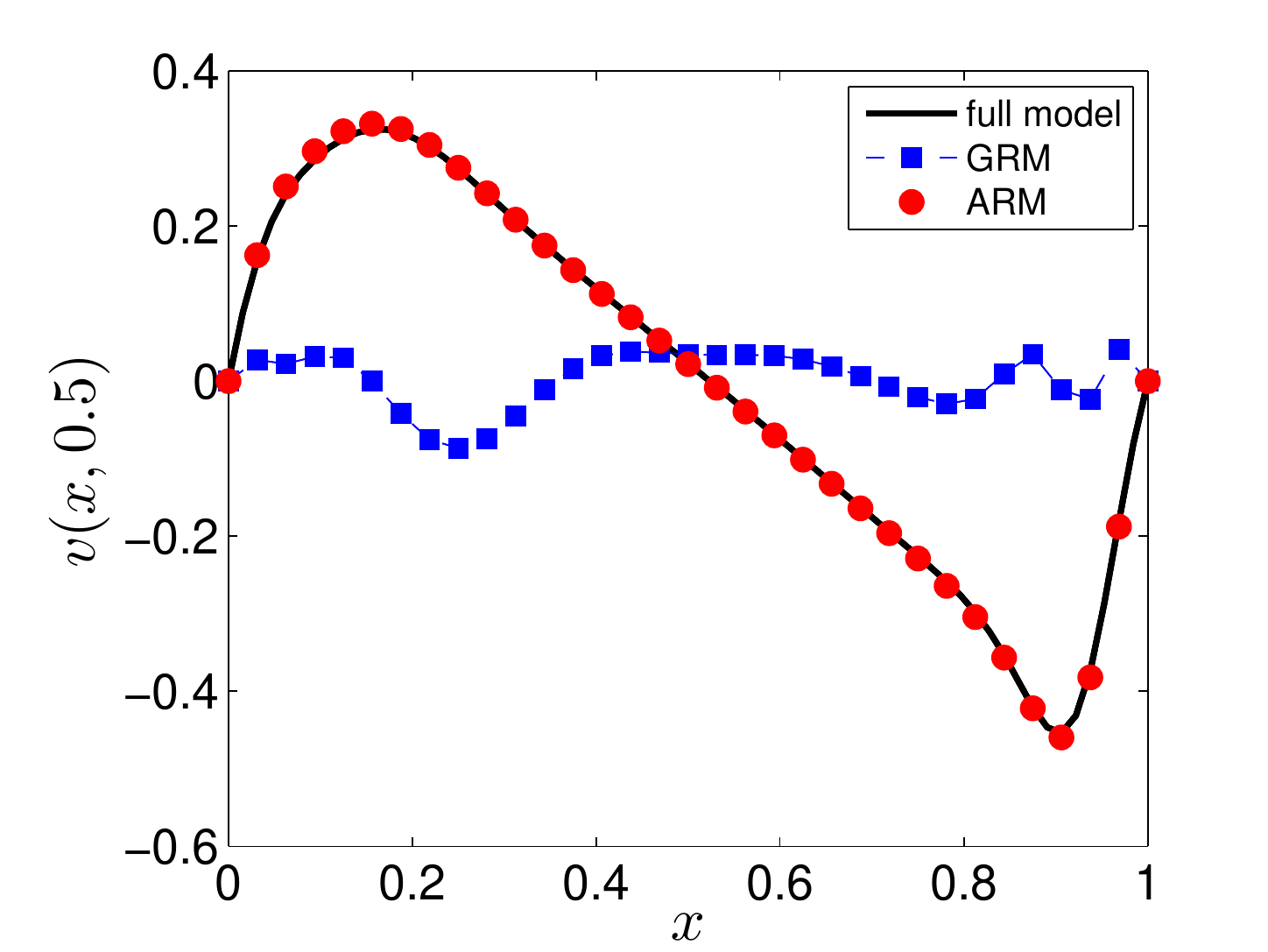}
\end{center} \end{minipage}
\begin{minipage}{0.48\linewidth}\begin{center} (a) \end{center}\end{minipage}
\begin{minipage}{0.48\linewidth}\begin{center} (b) \end{center}\end{minipage}
\caption{(Color online.) (a) Comparison of the velocity component $u(x=0.5, y)$
along the $y$-direction passing though the geometric center of the domain between the
full model, the GRM, and the ARM at $t=50$. The first $k=20$ modes are used for the GRM and the ARM.
(b) Comparison of the velocity component $v(x, y=0.5025)$ along the
$x$-direction passing through the geometric center of the domain between the full model, the GRM, and the ARM at $t=50$. } \label{fig:velocity}
\end{center}
\end{figure}

The streamline contours of Re=1050 and $L_y=1.05$ for the lid-driven cavity flow at $t=50$ are shown in Figure \ref{fig:stream}(a). The ARM provides an approximate solution with 20 modes, as seen in Figure \ref{fig:stream}(b). The numerical error, $u-\hat u$, of the ARM is shown in  \ref{fig:stream}(c).
Figure \ref{fig:velocity} illustrates the velocity profiles for $u$
along the vertical lines and $v$ along the horizontal lines passing
through the geometric center of the cavity.  The GRM
provides a poor approximation with 20 modes. But the ARM can yield much high accurate solutions with the same subspace dimensions.

Figure \ref{fig:burgerspod}(a) shows the  normalized error for 100 randomly selected new parameters. The direct projection error $e_k$ for both GRM and  ARM are also shown as a function of the subspace dimension $k$.  We observe that the GRM always has higher truncation error than the ARM for a fixed dimension.  Figure \ref{fig:burgerspod}(b) shows that the computational time of the GRM and the ARM are almost the same. When the subspace dimension is low, say $k\le 30$, both methods can obtain significant speedups.  In order to obtain a highly accurate representation for
the entire trajectory, the GRM needs a subspace with relatively high dimensionality to form a surrogate model. Thus, the  GRM is not necessarily  able to provide significant speedups for a dynamical system for long time evolution, especially when the subspace dimension is relatively high.  On the other hand, using the domain decomposition approach, the ARM can yield an accurate solution based on any prescribed subspace dimensions. Therefore, the ARM can always  obtain a significant speedup with good accuracy.

\begin{figure}
\begin{center}
\begin{minipage}{0.48\linewidth} \begin{center}
\includegraphics[width=1\linewidth]{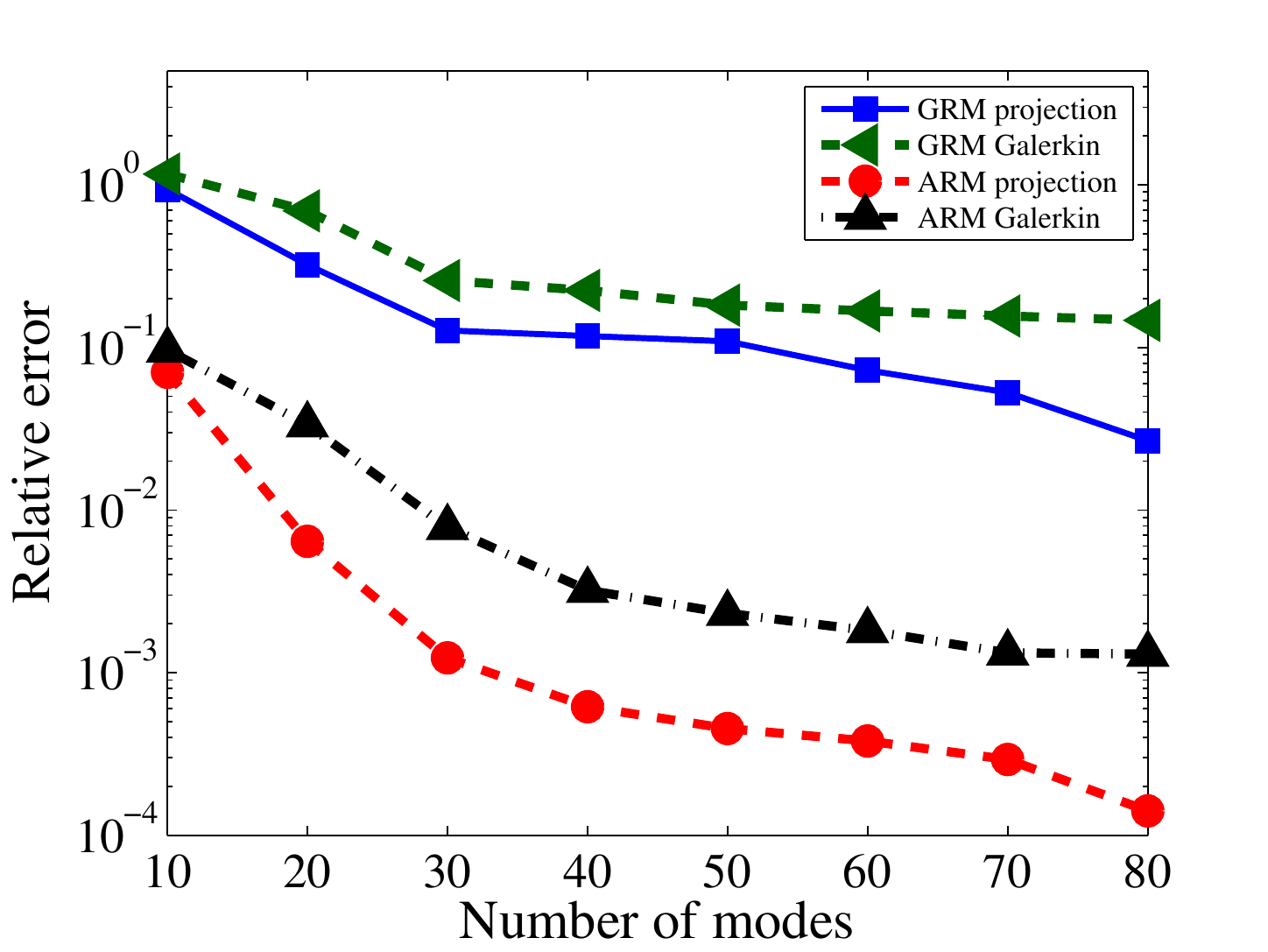}
\end{center} \end{minipage}
\begin{minipage}{0.48\linewidth} \begin{center}
\includegraphics[width=1\linewidth]{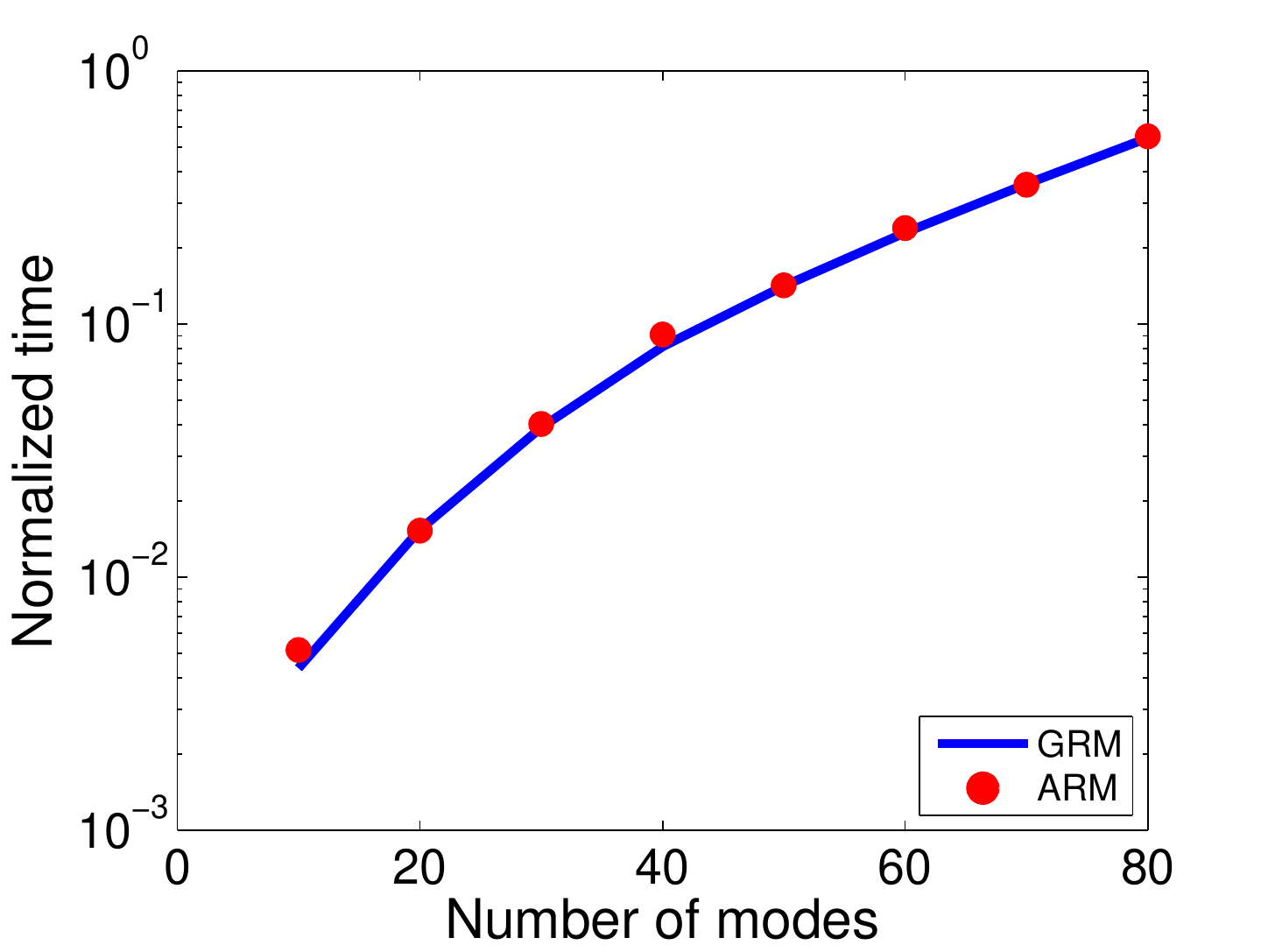}
\end{center} \end{minipage}\\
\begin{minipage}{0.48\linewidth}\begin{center} (a) \end{center}\end{minipage}
\begin{minipage}{0.48\linewidth}\begin{center} (b) \end{center}\end{minipage}
\caption{(Color online.) Simulation results for the 2D Navier-Stokes equation in a lid-driven cavity flow problem. (a)  The error for the ARM and the GRM obtained by the Galerkin projection. Plot of the relative error, $\|u(\mu)-\hat u(\mu)\|/\|u(\mu)\|$, based on 100 randomly selected parameters at $t=50$. For the same dimension, the ARM could always obtain higher accuracy compared with the GRM. (b) The relative running time of the GRM and the ARM. These values are normalized by the average running time of the full model.} \label{fig:burgerspod}
\vspace{-3mm}
\end{center}
\end{figure}

Finally, we study the numerical error of the  ARM Galerkin method. Except $k$ and $\sigma$, all the parameters
remain the same.
Table \ref{tab:psigma} indicates that the ARM error is not very sensible with $\sigma$; the subspace dimension $k$ is the dominant factor that determines the numerical error.

\begin{table} [htbp]
\begin{center}
\caption{The numerical error of the ARM chord method with different subspace $k$ and the kernel length $\sigma$.}
 \label{tab:psigma}
\begin{tabular}{r|ccccccccc}
& \multicolumn{9}{c}{$\sigma$} \\
\cline{2-10}
$k$ &	0.001&	0.002&	0.005&	0.01&	0.02&	0.05&	0.1&	0.2&	0.5\\
%\midrule [0.2pt]
\hline
20&	7.92E-4&	7.19E-4&	4.42E-4&	3.72E-4&	4.85E-4&	8.86E-4&    1.41E-3&    2.25E-3&	3.21E-3\\
40&	1.92E-4&	1.76E-4&	1.29E-4&	1.11E-4&	1.29E-4&	1.32E-4&	1.36E-4&	1.39E-4&	1.85E-4\\
60&	8.70E-5&	7.00E-5&	6.60E-5&	6.30E-5&	7.00E-5&	6.80E-5&	7.70E-5&	9.60E-5&	1.11E-4\\
\end{tabular}
\end{center}
\end{table}

\section{ Conclusion}\label{sec:conclusion}
In this paper, we have proposed a new technique, the adaptive reduced model (ARM), for model reduction of large-scale parameterized PDEs.  The method could be applied to both elliptic and parabolic PDEs based on the proper orthogonal decomposition (POD) and the discrete empirical interpolation method (DEIM). Compared with the global reduced model, The ARM could approximate the original system with much lower dimensions. Compared with the local reduced model, the  ARM can more  efficiently extract the information for all the precomputed snapshots; and therefore yield more accurate solutions. The ARM is especially suitable when the data snapshots are very expensive to obtain. For elliptic PDEs, the adaptive reduced basis can be  constructed by the singular value decomposition (SVD) of an information matrix, whose columns are given by weighted solution snapshots. Furthermore,  the reduced chord iteration could be used in the context of the ARM to save additional computational cost. For parabolic PDEs, an information matrix for the reduced basis can be constructed  from a set of weighted  empirical eigenfunctions, which has a smaller size compared with a matrix constructed from weighted snapshots. Both analytical results and numerical simulations demonstrate the capability of the ARM to solve a large-scale system with high accuracy and good efficiency.

\bibliographystyle{siam}
\bibliography{RefA3}

\end{document}